\documentclass[10pt]{amsart}
\usepackage{amssymb}
\usepackage{amsmath}
\usepackage{tikz}
\usepackage{amsfonts,amssymb}
\usepackage{mathrsfs}
\newtheorem{theorem}{Theorem}

\newtheorem{remark}{Remark}

\newtheorem{corollary}[theorem]{Corollary}

\newtheorem{definition}[theorem]{Definition}

\newtheorem{lemma}[theorem]{Lemma}
\newtheorem{proposition}[theorem]{Proposition}

\newtheorem{thm}{Theorem}[section]

\numberwithin{equation}{section} \numberwithin{theorem}{section}

\def\R{\mathbb R}

\def\Z{\mathbb Z}

\def\Z{\mathbb Z}

\def\lz{\lambda\cdot\mathbb Z^2}

\begin{document}

\title[Global wellposedness for the 3D KPII equation]{Global well-posedness and scattering for small data for the 3-D Kadomtsev-Petviashvili-II equation}

\author{Herbert Koch}

\address{Mathematisches Institut \\ Unversit\"at Bonn \\ Endenicher Allee 60 \\ 53115 Bonn \\ Germany}
\email{koch@math.uni-bonn.de}
\thanks{H.K has been partially supported by the DFG through CRC 1060}

\author{Junfeng Li}
\address{Laboratory of Math and
Complex Systems \\ Ministry of Education \\ School of Mathematical
Sciences \\ Beijing Normal University \\ Beijing 100875 \\  P. R. China}
\email{lijunfeng@bnu.edu.cn}
\thanks{J.L. has been partially supported by the Humboldt foundation,
  NSF of China (Grant No. 11171026), the Fundamental Research Funds for the Central Universities (NO. 2014KJJCA10).}

\begin{abstract} We  study  global well-posedness
  for the Kadomtsev-Petviashvili II equation in three space dimensions with
  small initial data.   The crucial points are new bilinear estimates and
 the definition  of the function spaces. As by-product we obtain that all solutions to
  small initial data scatter as $t \to \pm \infty$.
\end{abstract}


\maketitle


\vspace{0.25cm}

\section{Introduction and main results}

In this paper, we study the Cauchy problem for the  3-dimensional   Kadomtsev-Petviashvili II (KP-II) equation
\begin{align}\label{eq:1.1}
\left\{
\begin{aligned}
\partial_x\left(\partial_t u+\partial_x^3 u+\partial_x(u^2)\right)+\triangle_{y}u=& 0    \hspace{1.5cm}  & (t,x,y)\in \mathbb{R} \times \mathbb{R}\times \mathbb{R}^2  \\
u(0,x,y)=& u_0(x,y)  &  (x,y)\in \mathbb{R}\times \mathbb{R}^2.
\end{aligned}
\right.
\end{align}

The Kadomtsev-Petviashvili (KP) equations describe nonlinear wave interactions of almost
parallel waves. They come with at least four
different flavors: The KP-II equation for which the line soliton is
supposed to be stable, the KP-I equation with localized solitons, and
the modified KP-I and KP-II equations with cubic nonlinearities.

The KP-II equation is invariant under

\begin{enumerate}
\item Translations in $x,y$ and $t$.
\item Scaling: $\lambda^2 u( \lambda x, \lambda^2 y, \lambda^3 t) $ is  a solution if $u$ satisfies the KP-II equation \eqref{eq:1.1}.
\item Galilean transform: Let $c \in \mathbb{R}^2$. Then $ u(t, x- c
  \cdot y -|c|^2 t, y +2ct ) $ is a solution if $u$ satisfies
  \eqref{eq:1.1}. On the Fourier side the transform is $\hat u( \tau- |c|^2
  \xi-2c\cdot \eta, \xi, \eta +c\xi  ) $ where $\tau$ is the Fourier
  variable of $t$, $\xi\in \mathbb{R}$ is the Fourier variable of $x$
  and $\eta$ the one of $y$.
\item Isometries of the $y$ plane.
\item  Simultaneous reflections of $x$, $t$ and $u$.
\end{enumerate}

 The Galilean invariance is often a consequence of the
rotational symmetry of full systems for which certain solutions are
asymptotically described by a KP equation. The interest in the KP
equations comes from the expectation that they describe waves in a
certain asymptotic regime for a large class of problems, for which one
does not even have to formulate a full model, similar to the role of
the nonlinear Schr\"odinger equation in nonlinear optics.

The Galilean symmetry group is noncompact, in contrast to the
orthogonal group $O(n)$ and it seems that with this noncompactness the
difficulty increases with the dimension, in contrast to what is true
for many wave and Schr\"odinger equations. It would be interesting to
see whether the stronger decay of the linear equation compared to the 2d problem can be used to prove global existence for small Schwartz functions.

We search for spaces of initial data and solutions which reflect the
symmetries. Given $\lambda \in \mathbb{R} \backslash \{0\}$, we define
the Fourier projection $u_\lambda$ (we denote the Fourier transform by
$\mathcal{F}$ resp. $\hat{ }$) by
\begin{equation} \label{Fourierprojection}  \hat u_\lambda (\tau, \xi,\eta) = \left\{ \begin{array}{rl} \hat u(\tau, \xi,\eta) & \text{ if } \lambda \le|\xi| < 2\lambda \\ 0 & \text{otherwise.} \end{array} \right.   \end{equation}
We will always choose $\lambda$ to be a power of $2$.
For fixed $\lambda $, we partition the set
$\{ (\xi,\eta) \in \mathbb{R}\times \mathbb{R}^2: \lambda\le |\xi| < 2\lambda\}$ into sets $\Gamma_{\lambda,k}$ for $ k \in \lambda\cdot\mathbb{Z}^2$ defined by
\begin{equation} \label{eq:Gamma}  \Gamma_{\lambda,k} = \left\{ (\xi,\eta):   \lambda\le |\xi| < 2\lambda, \left| \frac{\eta}{\xi} - k \right|_{\infty}  \le \frac{\lambda}{2} \right\} \end{equation}
 where $ | a |_{\infty} = \max \{ |a_1|,|a_2|\}$.
This decomposition is shown below.

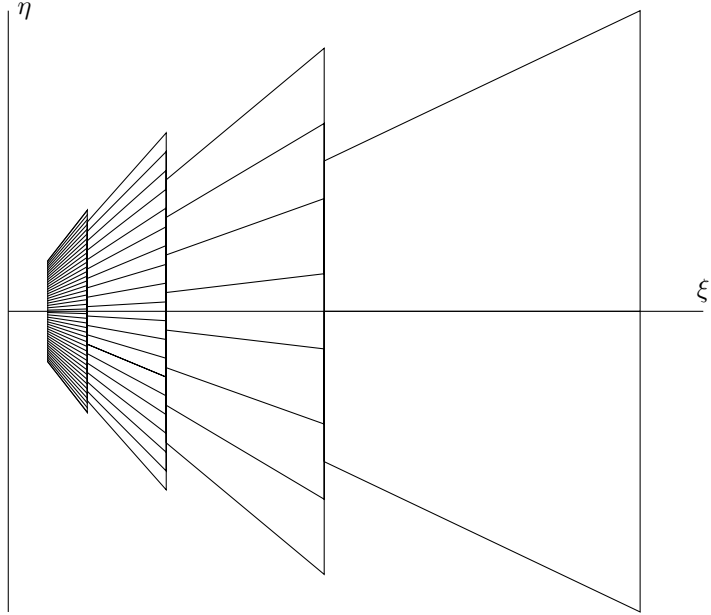
\begin{center}
\begin{tikzpicture}[xscale=4.2,yscale=.0625]
\draw(0,0)--(2.2,0) node[anchor=south]{$\xi$};
\draw(0,-64)--(0,64) node[anchor=west]{$\eta$};
\draw(1,-32)--(2,-64)--(2,64)--(1,32)--(1,-32)--(1,0)--(2,0);
\draw(0.5,-28)--(1,-56)--(1,56)--(0.5,28)--(0.5,-28)--(0.5,-12)--(1,-24)--(1,-8)--(.5,-4)--(.5,4)--(1,8)--(1,24)--(.5,12)--(.5,20)--(1,40)--(1,-40)--(.5,-20);
{%\clip (.249,-64) rectangle (1.01,64);
\draw(.25,-17)--(.5,-34)--(.5,34)--(.25,17)--(.25,-9)--(.5,-18)--(.5,-14)--(.25,-7)--(.25,-5)--(.5,-10)--(.5,-6)--(.25,-3)--(.25,-1)--(.5,-2)--(.5,2)
--(.25,1)--(.25,3)--(.5,6)--(.5,10)--(.25,5)--(.25,-7)--(.5,-14)--
(.5,14)--(.25,7)--(.25,9)--(.5,18)--(.5,-26)--(.25,-13)--(.25,-11)--(.5,-22)--(.5,22)--(.25,11)--(.25,13)--(.5,26)--(.5,30)--(.25,15)--(.25,-19)--(.5,-38)--(.5,38)--(.25,19)--(.25,-15)--(.5,-30);
\draw(.25,-19)--(.25,19);}
{ \clip (.124,-64) rectangle (.251,64);
\draw  (0,0)--(.5,43)--(.5,-43)--(0,0)--(.5,41)--(.5,-41)--(0,0)--(.5,39)
--(.5,-39)--(0,0)--(.5,37)--(.5,-37)--(0,0)--(.5,35)--(.5,-35)--(0,0)--(.5,33)
--(.5,-33)--(0,0)--(.5,31)--(.5,-31)--(0,0)--(.5,29)--(.5,-29)--(0,0)--(.5,27)
--(.5,-27)--(0,0)--(.5,25)--(.5,-25)--(0,0)--(.5,23)--(.5,-23)--(0,0)--(.5,21)
--(.5,-21)--(0,0)--(.5,19)--(.5,-19)--(0,0)--(.5,17)--(.5,-17)--(0,0)--(.5,15)
--(.5,-15)--(0,0)--(.5,13)--(.5,-13)--(0,0)--(.5,11)--(.5,-11)--(0,0)--(.5,9)
--(.5,-9)--(0,0)--(.5,7)--(.5,-7)--(0,0)--(.5,5)--(.5,-5)--(0,0)--(.5,3)
--(.5,-3)--(0,0)--(.5,-1)--(.5,1)--(0,0);
\draw (.125,10.75)--(.25,21.5)--(.25,-21.5)--(.125,-10.75)--(.125,10.75);
}
\end{tikzpicture}
\end{center}

\bigskip

For $1\leq q< \infty$, $1\leq p< \infty$,  a tempered distribution $f$ is said to be in $l^ql^pL^2$ if it is in the closure of $C^\infty_0$ with respect to the norm
\[\Vert f\Vert_{l^ql^pL^2}:=\left\{\sum_{\lambda\in 2^{\mathbb Z}}\lambda^{\frac{q}{2}}\left(\sum_{k\in\lambda\cdot\mathbb Z^2}\Vert f_{\Gamma_{\lambda,k}}\Vert^p_{L^2}\right)^\frac{q}{p}\right\}^{\frac{1}{q}}<\infty.\]
The case $p,q=\infty$ require the standard  modification. Here and in the sequel $f_{\Gamma_{\lambda,k}}$ denotes the Fourier projection.

We base our construction of the solution space on the space $V^2_{KP}$ of functions of bounded
2 variation $V^2$ adapted to the three dimensional KP-II equation. This
function space will be introduced in more detail in section \ref{bilinear}.
The solution space is defined as
\[
\|u\|_{l^q l^pV_{KP}^2}= \left(\sum_{\lambda\in 2^{\Z}}  \Big(\lambda^{\frac12} \sum_{k\in\lz}\|u_{\Gamma_{\lambda,k}}\|^p_{V^2_{KP}(\Gamma_{\lambda,k})}\Big)^{\frac{q}{p}}\right)^{\frac1q}<+\infty.
\]
We need also
 the homogeneous Fourier restriction space $\dot{X}^{0,b}$ for $|b|\le 1$
 which is defined by
\[
\|u_1\|_{\dot{X}^{0,b}}=\Vert |\partial_t-\partial_x^3+
\partial_x^{-1} \Delta_y|^b
u_1\Vert_{L^2}:=\||\tau-\omega(\xi,\eta)|^b\hat{u}_1\|_{L^2}<+\infty
\]
for tempered distributions supported in $[0,\infty) \times \R \times \R^2$.
Here $\omega(\xi,\eta)=\xi^3-\frac{|\eta|^2}{\xi}$ is the dispersion
function associated to KP-II equation. We define
\[
\|u\|_{l^q \dot{X}^{0,b}}= \Vert\lambda^2 u_{\lambda}(\lambda x,\lambda^2 y,\lambda^3 t)\Vert_{l^q_{\lambda}\dot{X}^{0,b}}=\left(\sum_{\lambda\in 2^{\mathbb{Z}}}\lambda^{(2-3b)q}
\Vert u_\lambda\Vert^q_{\dot{X}^{0,b}}\right)^{\frac1q}.\]
Here $l^p_\lambda$ denotes the $l^p$ norm with respect to the summation over
$\lambda \in 2^{\mathbb{Z}}$.
Finally we define the function space for the fixed point map by
\[
 \Vert u \Vert_X =\|u\|_{l^q l^p V^2_{KP} }+\|u\|_{l^q  \dot{X}^{0,b}}<\infty.
\]
Since $\sup_{t}  \Vert u(t)  \Vert_{L^2} \le \Vert u \Vert_{V^2_{KP}}$ (see \cite{KoBi})
one has
$ \sup_t \Vert u(t) \Vert_{l^q l^pL^2} \le \Vert u \Vert_X$.
It will be clear from the construction that we obtain  solutions in
$ u \in C([0, \infty); l^ql^pL^2),\,\,\text{for}\,\, 1\leq q<\infty, 1<p<2$.
We are ready  to state  our main results.

\begin{theorem}\label{wellposed} For $1\leq q<\infty$, $1<p<2$, there exists an $0<\varepsilon$ such that  if $u_0\in l^ql^pL^2$ satisfies
\[\Vert u_0 \Vert_{l^ql^pL^2}\leq\varepsilon\]
then there exist a unique global solution $w$ to \eqref{eq:1.1}
\[w=S(t)u_0+u\]
with $u\in X\subset C(\mathbb R, l^ql^pL^2).$  It satisfies
\begin{equation}\label{eq:1.2} \Vert u \Vert_X \le c  \Vert u_0  \Vert_{l^ql^pL^2}^2. \end{equation}
Here $S(t)u_0$ is the solution to the homogeneous problem defined by the Fourier transform (see \eqref{eq:2.1} in Section \ref{Strichartz} ).  The flow map
$$\Phi:B_{\varepsilon}\mapsto X: u_{0}\mapsto u \in X$$ is analytic.
Here the symbol $B_{\varepsilon}$ denotes the ball of radius $\varepsilon$
in $l^ql^pL^2$.
\end{theorem}

Scattering is an immediate consequence.
\begin{corollary}\label{scattering}[Scattering]  Under the assumption of Theorem  \ref{wellposed}  for  $u_0\in B_\varepsilon$
there exists
$u_{\pm}\in l^ql^pL^2 $ such that
\[
u(t)-S(t)u_{\pm} \rightarrow 0\text{\, in\, } l^ql^pL^2  \text{\,as\,} \, t\rightarrow\pm\infty.
\]

The  wave operators are the inverses of the maps
\[V_\pm :  B_\varepsilon\ni u_0 \rightarrow u_{\pm} \in   l^ql^pL^2. \]
They are analytic diffeomorphisms to their range if $\varepsilon $ is sufficiently small.
\end{corollary}

\begin{proof}  It is an important property of the spaces $V^2_{KP}$ that for
$ v \in V^2_{KP}$ the limit
\[ \lim_{t\to \infty}   S(-t) v(t) \]
exists. If $v \in X$ then $\lim_{t \to \infty} S(-t) v_{\Gamma_{\lambda,k}}$ exist. But then also
\[ \lim_{t\to\infty} S(-t) v(t) \]
exists in $l^ql^pL^2$. Since $ u_0 \to u(t) \in X$ is analytic
also the map $ u_0 \to \lim_{t\to \infty} S(-t) u(t) $ is analytic as
a function of $u_0$. Its derivative at $u_0=0$ is the identity, and
hence the map is invertible in a neighborhood of $u_0=0$.
\end{proof}

Theorem \ref{wellposed} is almost sharp.  For $2<p<\infty$, problem \eqref{eq:1.1} is ill-posed in the sense that the map $l^ql^pL^2 \ni u_0 \to u(t) \in l^ql^pL^2$ cannot be twice differentiable at $0$.
\begin{theorem}\label{illposed}
Let $1\leq q\leq\infty$, $2<p<\infty$. Suppose there exists  $T >0$
and $\varepsilon>0$
such that \eqref{eq:1.1} admits a unique solution defined on the interval $ [-T,T] $ for initial data in ball of radius $\varepsilon$ and center $0$ in $l^ql^pL^2$. Then  the flow map
\[F_t :u_0 \rightarrow u(t)\]
for \eqref{eq:1.1} is not twice differentiable at $u_0=0$  as a map from $l^ql^pL^2$ to itself.
\end{theorem}

We complement the results by studying the relation of the new function spaces
to test functions and distributions.
\begin{theorem}\label{discription}
For any $1\leq q\leq \infty$, we have
\begin{enumerate}
\item If $p<2$ then $l^q l^p L^2$ embeds
continuously into the space of distributions.
\item If $p\ge 2$ and $q>1$ there is a
sequence of Schwartz functions $\phi_j$ converging to $0$ in $l^q
l^p L^2 $, which  does not converge in the sense  of distributions.
\item  If $p\leq \frac43 $, and $ \phi$ is Schwartz function in
$l^q l^p L^2$ then for all $y \in \mathbb{R}^2$ we have
\[ \int \phi(x,y) dx = 0. \]
\item   The Schwartz functions are contained in $l^q l^p L^2$ if $\frac43 <p<\infty$.
\end{enumerate}

\end{theorem}

\begin{remark} For  $l^1l^2L^2=L^2(\R^2;B^{\frac{1}{2}}_{2,1})$ and  $l^2l^2L^2=\dot H^{\frac12,0}$ (see the definition \eqref{norms} below)
we do not know whether the flow map is smooth or not.
\end{remark}

It is worthwhile to compare our results to the 2-D KP II initial data
problem, which is much better understood.  It has the same symmetries
- up to obvious changes - as the three dimensional problem.
A scaling critical and Galilean invariant space is $\dot{H}^{-\frac12,0}$
defined by the norm
\begin{equation} \label{norms}  \Vert u_0 \Vert_{\dot{H}^{s,\sigma }} = \Vert |\xi|^{-1/2} <\eta>^\sigma \hat u_0\Vert_{L^2}.
\end{equation}
   In \cite{Bourgain3}, Bourgain settled the global well-posedness of
   the two dimensional version of \eqref{eq:1.1} in
   $L^2(\mathbb{R}^2)$. The assertion was then extended by Takaoka and
   Tzvetkov \cite{TkTz} (see also Isaza and Mej\'ia \cite{IsMe}) from
   $L^2(\mathbb{R}^2)$ to $H^{s_1,s_2}$ with $s_1>-\frac{1}{3},\,
   s_2\geq 0$. In \cite{Takaoka}, Takaoka obtained  local
   well-posedness for $s_1>-\frac{1}{2},\,s_2=0$ under an additional
   assumption on the low frequencies which was later removed by Hadac in
   \cite{Hadac1}. Hadac, Herr and the first author \cite{HaHeKo}
   studied  the two dimensional KP-II equation in the critical case $s_1=-\frac12,s_2=0$.  They
   obtained global well-posedness and scattering result in the
   homogeneous Sobolev space $\dot{H}^{-1/2,0}(\mathbb R^{2})$ with
   small initial data. A local well posedness result in $H^{-1/2,0}(\mathbb
   R^{2})$ was also obtained in \cite{HaHeKo}.  Some recent results on the KP-II equation can be found in \cite{KleSa}.

   Much less is known  for  KP II in three dimensional spaces.  Tzvetkov
   \cite{Tzvetkov} obtained  local well-posedness in $
   H^s(\mathbb{R}^3)$ with the additional condition
   $\partial_x^{-1}u\in H^{s}(\mathbb R^3)$ for $s>\frac32$. Here
   $H^s(\mathbb R^3)$ denotes the isotropic Sobolev space.
   Isaza, L\'opez and Mej\'ia \cite{IsLoMe} constructed unique local solutions
    in  Sobolev space $H^{s,r}(\mathbb R^3)$ defined by the norm
\[
\|f\|_{H^{s,r}(\mathbb R^3)}:=\|<\xi>^s<\zeta>^r\hat{f}(\zeta)\|_{L^2_{\zeta}}                \]
for $s,r \in \R$.  Hadac \cite{Hadac2} in his Ph.D thesis extended the
local well-posedness result to almost all the subcritical cases. He
obtained local well posed for \eqref{eq:1.1} in $Y_{s,r}(\mathbb R^3)$
for $s>\frac12, r>0$. To our best knowledge our result is the first
result for initial data in a scaling invariant space, and the first
scattering result for the three dimensional problem. Also the bilinear
estimates (Proposition \ref{th:2.1}) accounting for dispersion in $y$ seem to be new.

In the  3-D setting using the vertical direction (i.e. dispersion in the $y$ variable) is much
more important than in the two dimensional problem.  This can be see
from the Strichartz estimates in Theorem \ref{th:2.1} in Section
\ref{bilinear}. In particular the bilinear $L^4$ estimate by itself
seems not to suffice to close the iteration argument, and we need
several nontrivial modifications. In particular we use bilinear
estimates which give us a gain making  use of  the dispersion in $y$
direction.  We hope and think that these modifications and the
constructions are of interest beyond this particular problem at hand.
The 3D-KP II equation may be considered as a problem where the quadratic
nonlinearity satisfies a null condition which exactly balances
the bilinear estimates and the gain from high modulation, where we are not allowed to loose anything on the $L^2$ level.

The outline of this paper is following. In Section \ref{strichartz}
we prove the Strichartz estimates for the linear equations and a new
crucial and fundamental bilinear estimate, Theorem \ref{th:2.1}.  In
Section \ref{sketch} we give the proofs of our main results.
We first sketch an incorrect heuristic proof to show how far one gets using
simple bilinear estimates and high modulation, for $q=1$ and $p=2$.
A number of estimates is tight in this situation and we have not been able to
close the argument for those function spaces. In the remainder of this section
we sharpen the bilinear estimates and complete the proof of the main theorem.
In Section \ref{ill} we complete the paper by a proof of Theorem \ref{illposed}
and \ref{discription}.

% In Section \ref{higher regularity} we set up the global wellposedness in higher regularity spaces .

We use the standard notation $A\lesssim B$ to mean that there exists
constant $C>1$ such $A\leq C B$. Constants $C$ may differ from line to
line and depend on some obvious indices in the context but not on $A$
and $B$. $A\sim B$ means $\frac{1}{C} B\leq A\leq C B$.  Similarly we
denote $A\ll B$ for $A\leq\frac{1}{C} B$ for some $C>0$. The $s$ dimensional Hausdorff measure is denoted by $\mathcal{H}^s$ and its restriction to a set $S$
by $\mathcal{H}^s_S$.

\section{Strichartz estimates and bilinear refinements}
\label{strichartz}

\subsection{Strichartz estimate}\label{Strichartz}
 The linear equation
\[  u_t + u_{xxx} +\partial_x^{-1} u_{yy} = 0 \]
defines a unitary group $S(t)$ on $L^2$ by
\begin{equation} \label{eq:2.1}
 \mathcal{F} (S(t) u_0) =   e^{it (\xi^3-|\eta|^2/\xi)} \hat u_0.
\end{equation}

Given $u_0 $ the solution $u(t) = S(t) u_0$ satisfies the Strichartz estimates
of the next lemma. We denote by $|D_x|^s$ the Fourier multiplier
$|\xi|^s$, $\xi $ being as always the Fourier variable of $x$.

\begin{lemma} \label{le:2.1}Suppose that $2\le p\le \infty$ and
\begin{equation} \label{eq:2.2} \frac2p+\frac3q = \frac32. \end{equation}
 Then the following estimate holds for all $u_0 \in \mathcal{S}$
\[ \Vert   u \Vert_{L^p_t L^q_x} \lesssim   \Vert  |D_x|^{\frac1{3p}}    u_0 \Vert_{L^2}. \]
If $2\le q <\infty $
\begin{equation} \label{eq:2.3}\frac1p+\frac1q = \frac12 \end{equation}
then
\[ \Vert u \Vert_{L^p_t L^q_x} \lesssim \Vert |D_x|^{\frac2p} u_0  \Vert_{L^2}. \]
\end{lemma}

\begin{proof} We only sketch the proof. By a Littlewood Paley decomposition
(see \eqref{Fourierprojection})
and H\"older's inequality the estimate follows from
\[ \Vert u_{1} \Vert_{L^p_t L^q_x} \le c \Vert u_{1}(0) \Vert_{L^2} \]
for Strichartz pairs $(p,q)$  which in turn is a consequence of the calculation
of the complex Gaussian (as oscillatory integral)
\[\frac1{ 2\pi}  \int_{\R^2}  e^{i y\cdot \eta - i t \eta^2/\xi+it \xi^3} d\eta
= \frac{\xi}{4ti}  e^{i\frac{\xi|y|^2}{4t}+ i t\xi^3 }.
\]
By stationary phase and the lemma of van der Corput we obtain
\[   \left| \int \frac{\xi}{|\xi|} |\xi|^{1/2}  e^{i(x+ \frac{|y|^2}{4t}) \xi + i t\xi^3} d\xi    \right| \le C
|t|^{-\frac12}  \]
which we write as
\[ \Vert D_x^{\frac12} \mathcal{F}^{-1} e^{it (\xi^3-\eta^2/\xi)}
\Vert_{sup} \le C |t|^{-\frac32}. \] By complex interpolation, the
Hardy-Littlewood-Sobolev resp. weak Young inequality and a $T^*T$
argument \eqref{eq:2.2} follows.  The endpoint $p=2$ and $q=6$ follows
from \cite{KeelTao}.

The estimate
\[   \left| \int_{1\le |\xi|\le 2} \xi  e^{i(x+ \frac{|y|^2}{4t}) \xi + i t\xi^3} d\xi    \right| \le C   \]
is trivial. It leads to the second estimate \eqref{eq:2.3}
by the same standard arguments.
\end{proof}

It is remarkable that there is so much  flexibility in the choice of $p$
and $q$. This is true for the Schr\"odinger group, but there it comes
from a trivial combination of (sharp) Strichartz estimates with
Sobolev embedding. Here the situation is different due to the
unbounded $y$ direction.

\subsection{Bilinear estimates}

There is an important special case of \eqref{eq:2.3}:
\begin{equation} \label{eq:2.4}
 \Vert  u \Vert_{L^4(\R^4)} \le c \Vert |D_x|^{\frac12} u_0 \Vert_{L^2(\R^3)}.
\end{equation}

The proof of the main theorem relies crucially on the following bilinear refinements. We denote by $u_{<\mu}$ the Fourier projection to all $\xi$ frequencies
less in absolute value than $\mu$, by $u_{>\lambda}$  the Fourier projection to $\xi$ frequencies with absolute value $>\lambda$
 and by  $u_{\mu, \Gamma}$ the Fourier projection to
\[   \Big\{ (\xi,\eta):   \mu< |\xi| \le 2\mu,   \frac{\eta}\xi \in  \mu \Gamma \Big\}.  \]
Let $|\Gamma|$ denote the Lebesgue measure of $\Gamma$. With this notation  the following variant or sharpening of the bilinear estimate is true.

\begin{thm}\label{th:2.1} Let $0< \mu , \lambda$. Then
\begin{equation}\label{th:2.1a}
 \Vert u_{<\mu}   v_{>\lambda} \Vert_{L^2} \le c \mu  \Vert u_0 \Vert_{L^2} \Vert v_0 \Vert_{L^2},
\end{equation}
and, if $\mu \le \lambda$, if $\Gamma \subset \R^2$ is measurable, and if
either
\begin{itemize}
\item  $\mu \le \lambda/8$ or
\item $\lambda/8 < \mu \le \lambda$ and $\Gamma \subset  B_\lambda  (0) $ and the support of the Fourier transform of $v_\lambda$
is disjoint from $\R \times \R \times B_{10 \lambda^2}(0)$
\end{itemize}
 then
 \begin{equation}\label{th:2.1b}
\begin{split}
 \left\Vert  \int_{\R\times \R^2}  \Big(\lambda+ \left|\frac{\eta_1}{\xi_1} - \frac{\eta-\eta_1}{\xi-\xi_1} \right| \Big)   \hat u_{\mu,\Gamma}(t,\xi_1,\eta_1)  \hat v_\lambda(t,\xi-\xi_1,\eta-\eta_1) d\xi_1 d \eta_1  \right\Vert_{L^2}
&\\ & \hspace{-7cm}
\lesssim \mu |\Gamma|^{\frac12}
\Vert u_{0,\mu,\Gamma} \Vert_{L^2} \Vert v_{0,\lambda} \Vert_{L^2}.
\end{split}
\end{equation}

 \end{thm}

\begin{remark} Here as always $u_{0,\mu,\Gamma}$ denotes the Fourier projection of the initial data.
\end{remark}

\begin{remark} The condition for the second inequality is needed for
a bound of a derivative from below at a single point in the argument in
\eqref{lower} below. If $\mu \sim \lambda$, $ \Gamma = B_\lambda(0)$ and
the Fourier support of $v_\lambda$ is contained in $\R \times \R \times B_{10 \lambda^2}(0)$ then
there is no gain compared to the Strichartz estimate \eqref{eq:2.4}.
\end{remark}

\begin{proof}

We consider solutions to the  dispersive equation
\begin{equation}\label{eq:phi}   i \partial_t u + \phi(D) u = 0\end{equation}
with $\phi(D)$ defined as Fourier multiplier with a smooth real
function $\phi$. Then the Fourier transform of a solution with initial
data $u_0$ is a complex measure supported on the characteristic set
$\{(\tau,\xi): \tau= \phi(\xi)\}$.  Here we denote all spatial Fourier
variables by $\xi$.   If $u$ is the solution to \eqref{eq:phi} with
initial data $u_0$ then (essentially using a regularization and the
coarea formula to make sense of the calculus of Dirac measures)
\[ \hat u =    \hat u_0(\xi) \delta_{\Phi} =      \sqrt{2\pi} (1+2|\nabla \phi|^2)^{-1/2}   \hat u_0(\xi)     d\mathcal{H}^{d}|_{\Sigma}   \]
where $\Sigma = \{(\tau, \xi): \tau = \phi(\xi)\}$ is the characteristic set,
and
\[ \Vert \hat u \Vert_{L^2(\delta_\Phi)}= (2\pi)^{-1/2} \Vert u_0 \Vert_{L^2}. \]

By the formula of Plancherel
bilinear estimates for dispersive equations are equivalent to  $L^2$ estimates
of  convolutions of such signed measures supported in such surfaces.
 By the Cauchy-Schwarz inequality and the theorem of Fubini, for non-negative bounded measurable functions $h$ and $l$,
\[
\begin{split}
\Vert fh*gl \Vert_{L^2(\R^d)}^2  \hspace{-1.5cm}& \\ = &
\int_{\R^{d}}
\left(\int_{\R^{d}}     f(x)(h(x) l(z-x))^{1/2}    g(z-x)  (h(x)  l(z-x))^{1/2}  dx \right)^2 dz \\
\le & \int_{\R^{d}}  \int_{\R^d} f^2(x)h(x) l(z-x) dx   \int_{\R^d}
g^2(y) h(z-y)l(y) dy \, dz \\
\le &   \int_{\R^{2d}} \left[ \int_{\R^d} h(z-y) l(z-x) dz \right]   f^2(x) h(x)  g^2(y) l(y) dx dy.
\end{split}
\]
Suppose that $U,V \subset \R^d$ are open,  $\Phi_{1} \in C^1(U)$, $\Phi_2 \in C^1(V)$ and  that the gradients $\nabla \Phi_i$ are nonzero where $\Phi_i$ vanishes. We define the Dirac measures $\delta_{\Phi_i}$
by approximation. The zero set of $\Phi_i$ is denoted by $\Sigma_i$.
The calculation above yields
\[ \Vert f\delta_{\Phi_1} * g\delta_{\Phi_2} \Vert_{L^2}
\le  C \Vert f \Vert_{L^2(\delta_{\Phi_1})} \Vert g \Vert_{L^2(\delta_{\Phi_2})} \]
where
\begin{equation}\label{eq:C}   C^2 = \sup_{x\in \Sigma_1, y \in \Sigma_2}   \int_{\R^d} \delta_{\Phi_2(z-x)}
\delta_{\Phi_1(z-y)} dz \end{equation}
which has again to be understood as limit through the approximation of the Dirac measures by smooth
functions.  By the coarea formula the integral can be rewritten. Let
\[ \Sigma_{x,y} = \{ z \in \R^d : z-x \in \Sigma_2, z-y \in \Sigma_1 \}
= (x+\Sigma_2) \cap (y+\Sigma_1).\]
With
\[  D= \left( \begin{matrix} d\Phi_1(z-x)  \\ d\Phi_2(z-y)  \end{matrix}  \right) \]
\[ J(z,x,y)  = \left( \det( D^T D) \right)^{1/2}, \]
we have
\begin{equation}\label{eq:2.5}   C^2 = \sup_{x,y}   \int_{\Sigma_{x,y}}  J(x,y,z) d\mathcal{H}^{d-2}(z).
\end{equation}
The case $ \Phi_i(\tau,\xi) = \tau-\phi(\xi)$, but with
$\Phi_1$ defined on $\R \times A $ and $\Phi_2$ on $\R \times B$
 is of particular interest.
Integrating  out $\tau$ \eqref{eq:C} simplifies to (with $n=d-1$)
\begin{equation}\label{c2}
C^2 = \sup_{\xi_1\in A,\xi_2\in B} \sup_{\tau}   \int_{\R^n}   \delta_{\phi(\xi-\xi_1)-\phi(\xi-\xi_2)-\tau} d\xi.
\end{equation}

The first case of interest is $U= \{(\tau,\xi,\eta):  |\xi|\le \mu\}$,
$V =    \{(\tau,\xi,\eta): \lambda \le |\xi| \}$ and
\[ \phi=\phi_1 = \phi_2 = \xi^3 -|\eta|^2/\xi. \]
To obtain the bilinear estimate \eqref{th:2.1a} we have to estimate the integrals in \eqref{c2}
by a constant times $\mu^2$. By the $L^4$ estimate \eqref{eq:2.4} we may assume that $\mu \le \lambda/2$ and  estimate the quantity in \eqref{eq:C}:
\begin{equation} \label{c2kp}
 C^2 = \sup_{\tau,\xi_1,\xi_2,\eta_1,\eta_2}
 \int_{\R^3,\xi-\xi_2 \in A, \xi-\xi_1\in B}
  \delta_{\phi(\xi-\xi_1,\eta-\eta_1)-\phi(\xi-\xi_2,\eta-\eta_2)-\tau}   d\eta   \,  d\xi.
\end{equation}
 The algebraic identity
\begin{equation}\label{eq:2.6}
\begin{split}
 \phi(\xi-\xi_1,\eta-\eta_1) - &  \Phi(\xi-\xi_2,\eta-\eta_2) +\Phi(\xi_1-\xi_2,\eta_1-\eta_2  )
\\ & + 3(\xi_1-\xi_2)(\xi-\xi_1)(\xi-\xi_2)   \\ = & (\xi_2-\xi_1) (\xi-\xi_1)(\xi-\xi_2)
\left( \frac{\left|\frac{\eta-\eta_1}{\xi-\xi_1} - \frac{\eta-\eta_2}{\xi-\xi_2}\right|}{|\xi_1-\xi_2|} \right)^2
\end{split}
\end{equation}
can be verified by  an easy  calculation. In particular, if we fix $\xi$
then either the $\eta$  integral is over the empty set, a point, or
it is an integral over a circle, in which case by \eqref{eq:2.6}
(it suffices to consider the coefficient of the quadratic term since the integral is independent of the radius)
 \[ \int_{\R^2}
  \delta_{\phi(\xi-\xi_1,\eta-\eta_1)-\phi(\xi-\xi_2,\eta-\eta_2)-\tau}   d\eta   = \frac{4\pi |\xi_2-\xi_1|}{|\xi-\xi_1||\xi-\xi_2|} \]
and we estimate the integral with respect to $\xi$ for $\mu \le \lambda/2$
\[  \frac{2\pi} {|\xi_2-\xi_1|}
 \int_{|\xi-\xi_2| \le \mu} |\xi-\xi_2|  |\xi-\xi_1|  d\xi \le 8  \pi \mu^2 .   \]
Together with the $L^4$ Strichartz estimate this implies estimate
\eqref{th:2.1a}.

We turn to the second part, \eqref{th:2.1b}, for which we repeat the
calculus argument. Here we want to recover the stronger bilinear
estimate for the KP equation where one  gains a full derivative. Of
course this can only be done by reducing the domain of the
integration. The final integration then leads to the factor given by
measure of $|\Gamma|$.

Let $ \Phi_i$ be as above. Instead of estimating  the convolution itself we claim that
\[
\begin{split}
\left\Vert \int h(y ,x -y) f_1(y) f_2(x-y)\delta_{\Phi_1}(y)
\delta_{\Phi_2}(x-y)   dy   \right\Vert_{L^2}
& \\ & \hspace{-4cm}  \le   C   \Vert f_1 \Vert_{L^2(\delta_{\Phi_1})}  \Vert  f_2 \Vert_{L^2(\delta_{\Phi_2})}
\end{split}
 \]
where
\[ C^2 =  \sup_{x \in \Sigma_1, y \in \Sigma_2} \int  h^2(z-x, z-y) \delta_{\Phi_1(z-x)} \delta_{\Phi_2(z-y)}   dz.   \]
This follows by the same calculation as above.

We take up the bilinear estimate for the KPII equation and estimate
the integral in \eqref{c2kp} with the integration restricted to a suitable set.
We  fix $\tau$, $\xi_1$, $\xi_2$, $\eta_1$ and $\eta_2$.
 We search an estimate
which contains the measure of $\Gamma$ and  apply the transformation formula
and Fubini's theorem to take the integration with respect to $\Gamma$ as
outer integration. This yields the desired estimate provided
we get uniform bounds  for the integral with respect to  $\xi$ for
$\frac{\eta-\eta_2}{\xi-\xi_2} = \rho \in \R^2$ fixed.
The Jacobian determinant of the map
\[ (\xi, \eta) \to (\xi, \frac{\eta-\eta_2}{\xi-\xi_2} ) \]
from $\R^3$ to $\R^3$ is $\frac{1}{|\xi-\xi_2|^2} $.
We assume that one of the conditions of the second part of the theorem holds.
Let   $h=  \lambda + \left|\frac{\eta_1}{\xi_1} - \frac{\eta_2}{\xi_2}\right| $ be the integrand to be studied. We recall that  $\Gamma \subset \mathbb{R}^2$ and
denote
 \[ B = \left\{(\xi,\eta): \mu/2 \le |\xi|\le \mu,  \frac{\eta-\eta_2}{\xi-\xi_2} \in \mu \Gamma \right\}.  \]
Then
 \[
\begin{split}
\int_B
    \Big(\lambda + \Big| \frac{\eta-\eta_1}{\xi-\xi_1}  - \frac{\eta-\eta_2}{\xi-\xi_2}\Big|\Big)^2
\delta_{\phi(\xi-\xi_1,\eta-\eta_1)-\phi(\xi-\xi_2,\eta-\eta_2)}  d\xi d\eta  \hspace{-8cm}  & \hspace{8cm}
\\  = &  \int_{\Gamma}  \int  \Big(\lambda + \Big| \frac{\eta-\eta_1}{\xi-\xi_1}  - \frac{\eta-\eta_2}{\xi-\xi_2}\Big|\Big)^2   |\xi-\xi_2|^2
 \delta_{g_\rho}(\xi)    d\xi   d\gamma
\\   \le & C \mu^2|\Gamma|
\end{split}
\]
where we calculated with
\[
\begin{split}
  (\xi-\xi_1)^3 & - \frac{(\eta-\eta_1)^2}{\xi-\xi_1} - (\xi-\xi_2)^3 +  \frac{(\eta-\eta_2)^2}{\xi-\xi_1}
\\ = & (\xi-\xi_1)^3 - \frac{( \rho\cdot (\xi-\xi_2) + \eta_2-\eta_1)^2}{\xi-\xi_1}          - (\xi-\xi_2)^3      + (\xi-\xi_2) |\rho|^2
\\ = :& g_\rho(\xi).
\end{split}
\]
Clearly $g_\rho(\xi)=\tau$ if and only if
\[ (\xi-\xi_1)^4 - (\rho\cdot (\xi-\xi_2) +\eta_2-\eta_1)^2-(\xi -\xi_2)^3(\xi-\xi_1)
+ (\xi-\xi_1)(\xi-\xi_2)|\rho|^2 = \tau \]
and hence there are at most $4$ values of $\xi$ where $g_\rho=\tau$.
Moreover
\begin{equation} \begin{split}  \left|\frac{d}{d\xi} g(\xi)\right| = & \Big| 3(\xi-\xi_1)^2 - 2 \rho \frac{\rho (\xi-\xi_2)+\eta_2-\eta_1}{\xi-\xi_1}\\ &  +  \frac{( \rho(\xi-\xi_2) + \eta_2-\eta_1)^2}{(\xi-\xi_1)^2}
- 3(\xi-\xi_2)^2 + |\rho|^2 \Big|\\
= &\left|  3 (\xi-\xi_1)^2 - 3 (\xi-\xi_2)^2
+ \left|\frac{\eta-\eta_2}{\xi-\xi_2} - \frac{\eta-\eta_1}{\xi-\xi_1} \right|^2
\right|
\\  \sim & \, \Big(\lambda + \Big| \frac{\eta-\eta_1}{\xi-\xi_1}  - \frac{\eta-\eta_2}{\xi-\xi_2}\Big|\Big)^2
\end{split}
\label{lower}
\end{equation}
since $g_\rho(\xi)=\tau $ at most at four points, and it satisfies the lower bound there.
\end{proof}

\subsection{Functions of bounded $p$ variation and
their predual}

Functions of bounded $p$ variation were introduced by
N.Wiener \cite{Wiener}. The space of function of bounded $p$ variation
and their pre-dual spaces $U^p$ were defined by D.Tataru and the first
author of this paper in \cite{KoTa}.  $V^p_{KP}$ and $U^p_{KP}$ are   defined by $S(t)V^p$ and $S(t)U^p$. Here $S(t)$ is the
unitary group defined in \eqref{eq:2.1}.  We refer the reader to
\cite{HaHeKo} for the following statements and further properties
about $U^p_{KP}$ and $V^p_{KP}$. Let $\frac1p+\frac1{p'}=1$, $1<p < \infty$. The duality pairing  can formally be written as
\[    B(u,v) = \int v (\partial_t+\partial_{xxx} - \partial_x^{-1} \Delta_y)  \bar u dx dy  dt, \]
 but a correct definition requires more care (see \cite{HaHeKo1}). The space $V^{p'}_{KP}$ is the dual  of $U^p_{KP}$ with respect to this duality pairing.
We denote by $V^p_{rc}$ the subspace of $V^p_{KP}$ of right continuous functions
with limit $0$ at $-\infty$.

The spaces $U^p$ have an atomic structure and
the Strichartz estimates imply
\begin{equation}\label{eq:2.7}  \Vert u \Vert_{L^pL^q} \le c_1\Vert |D_x|^{\frac1{3p}} u \Vert_{U^p_{KP}}  \end{equation}
where $\frac2p + \frac3q = \frac32$, $2\le p,q \le \infty$ and
\begin{equation}\label{eq:2.8}  \Vert u \Vert_{L^pL^q} \le \Vert D^{\frac1p} u \Vert_{U^p_{KP}} \end{equation}
where $\frac1p+\frac1q = \frac12$, $2<p\le \infty$. Moreover one has the inclusions
\begin{equation}\label{eq:2.9}  \Vert u \Vert_{U^p_{KP} } \le c \Vert u \Vert_{V^q_{KP} }
\end{equation}
whenever $q<p$ and $u \in V^q_{KP}$ is right continuous.   Similarly we obtain from the bilinear estimates of Theorem
\ref{th:2.1} under the same assumptions there,
\begin{equation} \label{eq:2.10}
\Vert u_\mu v_\lambda \Vert_{L^2} \le c \mu \Vert u_{\mu} \Vert_{U^2_{KP}}
\Vert v_\lambda \Vert_{U^2_{KP}}
\end{equation}
and
\begin{equation}\label{eq:2.11}  \left\Vert  \int_S \Big(\lambda+ \left|\frac{\eta_1}{\xi_1} - \frac{\eta_2}{\xi_2} \right| \Big)   \hat u_{\mu,\Gamma} \hat v_\lambda \right\Vert_{L^2} \lesssim \mu |\Gamma|^{\frac12}
\Vert u_{\mu,\Gamma} \Vert_{U^2_{KP}} \Vert v_{\lambda} \Vert_{U^2_{KP}  }. \end{equation}
The $V^2_{KP}$ spaces behave well with respect to further decompositions:
\begin{equation} \Vert u_\lambda \Vert_{V^2_{KP}} \le  \Vert u_\lambda \Vert_{l^2 V^2_{KP}}, \end{equation}
see \cite{KoBi}.
They allow the following decomposition

\begin{lemma}
\label{interpo}
 Suppose that $1<p<q<\infty$. There exists $\delta>0$ so that for
any right continuous $v \in V^p_{KP}$ and $M>1$ there exists $u \in U^p_{KP}$
and $w \in U^q_{KP}$ such that
\[ v = u + w \]
\[ \Vert u \Vert_{U^p_{KP}} \le M , \qquad \Vert w \Vert_{U^q_{KP}} \le e^{-\delta M }.
\]
\end{lemma}

From \eqref{eq:2.10},   the $L^4$ Strichartz estimates and logarithmic interpolation lemma \ref{interpo}
(see again \cite{HaHeKo}), we  obtain  for any $0<\varepsilon\ll 1$,
\begin{equation}\label{eq:2.12}
\Vert u_\mu v_\lambda \Vert_{L^2} \le C(\varepsilon) \mu \Big(\frac{\lambda}{\mu}\Big)^{\varepsilon}\Vert u_{\mu} \Vert_{V^2_{KP}}
\Vert v_\lambda \Vert_{V^2_{KP}}.
\end{equation}

Similarly the bilinear estimate \eqref{th:2.1b} implies bilinear estimates with respect to $U^2_{KP}$, and via logarithmic interpolation,
estimate with respect to the $V^2_{KP}$ norm.

Later we will make use of the spaces $U^1_{KP}\subset V^1_{KP}$ which
carry identical norms, which, for functions given by $S(-t)u(t) =
\int_{-\infty}^t f(s) ds $ is $\int_{\R} |f| dt$.
We define
\[
\begin{split}
 \Vert v \Vert_{V^1_{KP}}= & \Vert S(-t) v(t) \Vert_{BV}
\\  = &  \sup_{t_0  < t_2 <\cdots< t_n} \sum_{j=1}^n \Vert S(-t_i) v(t_i) - S(-t_{i-1})
v(t_{i-1}) \Vert_{L^2}
\end{split}  \]
where we allow $t_n= \infty$ (recall the convention  $v(\infty)=0$).  We denote by $U^1_{KP}$ the Banach space
of all right continuous functions with $\lim_{t\to -\infty} u(t) = 0 $
for which this norm is finite.  It is not hard to see that
\[ \Vert u \Vert_{U^1_{KP}}= \Vert  S(-t) u(t) \Vert_{BV(\R,L^2)} \]
Then $U^1_{KP} \subset U^2_{KP}$. We
will use an improvement of the estimate for high modulation.  Let
$\Phi \in \mathcal{S}(\R)$ with $\hat \Phi = 1$ for $|\tau|\le 1$,
$\hat \Phi=0$ for $|\tau| \ge 2$. Then, for $f$ with $f' \in L^1$
\[
\begin{split}
\Vert f-\Phi*f \Vert_{L^1}
= & \left\|\int (f(t)-f(s))\Phi(t-s) ds \right\|_{L^1}
\\ \le & \int_{\R}  |\Phi(\sigma)| \int |f(t)-f(t-\sigma)| dt d\sigma
\\ = & \int_{\R} |\sigma| |\Phi(\sigma)| d\sigma \int |f'(t)| dt.
\end{split}
\]
Rescaling and an approximation yield the high modulation estimate
\begin{equation}  \Vert   u_\lambda^{>\Lambda} \Vert_{L^1_t L^2} \le c\Lambda^{-1}
\Vert u_\lambda \Vert_{U^1_{KP}}. \end{equation}

Here $u^{>\Lambda}$ resp $u^{\leq\Lambda}$  means the
Fourier projection
to high resp- low modulation,
 i.e. to
$$|\tau-\omega(\xi,\eta)|:=\Big|\tau-(\xi^3-\frac{|\eta|^2}{\xi})\Big| > \Lambda $$
resp. $\leq \Lambda$.  By the definition of the Fourier restriction spaces
\[ \Vert u^{>\Lambda}  \Vert_{L^2} \le \Lambda^{-b} \Vert u^{>\Lambda}  \Vert_{\dot{X}^{0,b}}, \qquad
 \Vert u^{>\Lambda}  \Vert_{L^2} \le \Lambda^{-1/2} \Vert u^{>\Lambda}  \Vert_{V^2_{KP}}, \]
and similarly
\[   \Vert u^{\sim \Lambda} \Vert_{U^2_{KP}} \le \Lambda^{1/2} \Vert u \Vert_{L^2}. \]
see \cite{HaHeKo}.

\subsection{A bilinear operator}

\label{bilinear}

The bilinear estimates of Theorem \ref{th:2.1} state some off-diagonal
decay in the bilinear terms.
This suggests to decompose waves into wave packets of corresponding Fourier support.
We recall that we partition $\{\lambda, 2\lambda\}\times \R^2$ into sets
$ \Gamma_{\lambda,k}$ \eqref{eq:Gamma}.
 Theorem \ref{th:2.1} effectively diagonalizes the bilinear estimate
in the sector determined by the large frequency. To capture this we define
\[ \Gamma_{\lambda,k,L}= \left\{ (\xi_1,\eta_1) : \lambda \le \xi_1 \le 2\lambda,
\Big|\frac{\eta_1}{\xi_1}-kL \Big|_\infty \le  \frac{L\lambda}2
 \right\} \]
and $\Gamma_{\mu,k, L\lambda/\mu}$ is the set in frequency $|\xi| \sim \mu$
which corresponds to $\Gamma_{\lambda,k}$ in the bilinear estimate of Theorem
\ref{th:2.1}. We define a smooth bilinear projection which is compatible
with scaling and the Galilean symmetry.
Here we again denote the Fourier transform in space time by $\mathcal{F}$
resp. $\hat {\, }$. Let $ \phi_1 \in C^\infty_0((-129,129)\times (-129,129))$, identically $1$
in $(-128,128)\times (-128,128)$ and even. We define
for $L = 2^k$ with $k \ge 1$
\[ \psi_L (s) = \phi_1( s/L)  )- \phi_1(2s/L) \]
and
\[ \rho_L(\xi_1,\eta_1, \xi_2,\eta_2) :=
\psi_L\left( \frac{\frac{\eta_1}{\xi_1}-\frac{\eta_2}{\xi_2}}{\xi_1+\xi_2}    \right). \] For $L=1$, we make the modification
\[
\rho_1(\xi_1,\eta_1,\xi_2,\eta_2):=\phi_1\left( \frac{\frac{\eta_1}{\xi_1}-\frac{\eta_2}{\xi_2}}{\xi_1+\xi_2}    \right).
\]

 \begin{definition} We define the bilinear operators by their Fourier transform
\[  \! \mathcal{F}( T_L(v_\mu ,u_\lambda ))(\tau,\xi,\eta) = \int_{S}
\rho_L(\xi_1, \eta_1,\xi_2,\eta_2)
\hat v_\mu(\tau_1,\xi_1, \eta_1) \hat u_\lambda(\tau_2,\xi_2,\eta_2)
d\mathcal{H}^4.\!  \]
Here $S=\{\xi=\xi_1+\xi_2,\eta=\eta_1+\eta_2,\tau=\tau_1+\tau_2\}$ and $d\mathcal{H}^4$ denotes the 4-Dimensional Hausdorff measure on it.
\end{definition}
The product is the dyadic sum of these bilinear operators.
The key properties of the bilinear projection are its symmetry, and
the bounds of Proposition \ref{pro:3.1} below.

\begin{lemma}\label{TL} The following symmetry identity always holds.
\begin{equation}
\label{T:sym} \begin{split}  \int u_\lambda T_L (v_\lambda ,w_\mu) dx\, dy\, dt = &
 \int v_\lambda T_L(u_\lambda, w_\mu) dx\, dy \, dt
\\ = & \int w_\mu  T_L(u_\lambda, v_\lambda) dx\, dy \, dt.
\end{split}
 \end{equation}
\end{lemma}
\begin{proof}
This follows from the algebraic calculation
\[ \frac{\frac{\eta_1+\eta_2}{\xi_1+\xi_2} - \frac{\eta_1}{\xi_1} }{\xi_2} =
\frac{\frac{\eta_2}{\xi_2}-\frac{\eta_1}{\xi_1}}{\xi_1+\xi_2}. \]
\end{proof}

The following bilinear estimates provide us with a crucial new tool.
Below the index $\{.\}_+$ denotes the positive part.

\begin{proposition}\label{pro:3.1}
Let $\varepsilon >0$, $1\le p ,q,r \le \infty$ with
\[ \frac1r \le \frac 1p+\frac1q  \]
and $L \in 2^{k}, k=0,1,2\dots $. Then the following estimates hold
\begin{equation} \label{eq:3.10}
\left\Vert  T_L (u_\mu ,v_\lambda)_\lambda \right\Vert_{l^r L^2} \le C \mu
\Big(\frac{L\lambda}{\mu}\Big)^{1-\frac2p+\varepsilon} L^{(1-\frac2q)_++ (\frac2r-1)_+}
    \Vert u_\mu \Vert_{l^pV^2_{KP}} \Vert v_\lambda \Vert_{l^{q}V^2_{KP}}
\end{equation}
and
\begin{equation} \label{eq:3.15}\begin{split}
\Vert (T_L (u_\lambda ,v_\lambda))_\mu  \Vert_{l^rL^2} &
\\ & \hspace{-2cm}
 \le   C \lambda
\Big(\frac{L\lambda}{\mu}\Big)^{(\frac2r-1)_+}
L^{(1-\frac{2}{p}) +(1-\frac2q)_++\varepsilon}
\Vert u_\lambda \Vert_{l^{p}V^2_{KP}} \Vert v_\lambda \Vert_{l^{q}V^2_{KP}} .
\end{split}
\end{equation}
\end{proposition}

\begin{proof}
  We consider the  case $\mu < \lambda/4$ for \eqref{eq:3.10} first.
By rescaling we may assume that $\mu=1 <   \lambda/4$.
  We decompose the bilinear term further, using that by the definition of
  $T_L$ there is  only a contribution if
\[   \left| \frac{\frac{\eta_1}{\xi_1} -\frac{\eta_2}{\xi_2}}{\xi_1+\xi_2} \right| \sim L. \]
It is important that this relation is equivalent to
\[   \left| \frac{\frac{\eta_1+\eta_2}{\xi_1+\xi_2} -\frac{\eta_2}{\xi_2}}{\xi_1} \right| \sim L. \]

Since $1\le \lambda/4$ we have  $|\xi_2| \sim |\xi_1+\xi_2 |\sim \lambda $
and both $\xi_2$ and $\xi_1+\xi_2$ have the same sign. For simplicity
we assume that both are positive. Recall that $|\xi_1| \sim 1$. We begin with
the case $L=1$ resp.
\[ \left| \frac{\frac{\eta_1}{\xi_1}
    -\frac{\eta_2}{\xi_2}}{\xi_1+\xi_2} \right| \le 1. \]
If
$(\xi_2,\eta_2) \in \Gamma_{\lambda,k}$ then the $l^r$ summation in \eqref{eq:3.10} over  $\Gamma_{\lambda,l}$
contributes only if $|k-l| \le C$.  We simplify our lifes and restrict
to $l=k$. The situation
is similar if $L >1$.
and we obtain the restriction that the indices are of distance $\sim$ $1$
and the slopes have distance $\sim L \lambda $.

Hence, by the same abuse of notation as usual, and with the sets $\Gamma_{1,k,\lambda}$ defined at the beginning of this subsection
 \begin{equation}\label{eq:5.9}
(T_{1}(u_1,v_\lambda))_{\Gamma_{\lambda, k}}=
\left(T_1(u_{\Gamma_{1,k,\lambda}},v_{\Gamma_{\lambda,k}})\right)_{\Gamma_{\lambda, k}} .
 \end{equation}
We search for an $L^2$ estimate and  ignore the outer restriction to
$\Gamma_{\lambda,k}$ in the notation.
 By the bilinear estimate we get
\[ \Vert u_{\Gamma_{1,l}} v_{\Gamma_{\lambda ,k}} \Vert_{L^2}
\le c \frac{1}{\lambda}  \Vert u_{\Gamma_{1,l}} \Vert_{U^2_{KP}}\Vert  v_{\Gamma_{\lambda,k}} \Vert_{U^2_{KP}}.
\]
There are $\sim \lambda^2$ such terms in $u_{\Gamma_{1,k,\lambda}}$ contributing to the sum and hence by H\"older's inequality applied to the finite sum
\[ \Vert   u_{\Gamma_{1,k,\lambda}}v_{\Gamma_{\lambda,k}}\Vert_{L^2}
\le c \frac{1}{\lambda} \lambda^{2-\frac2p}
\Vert u_{\Gamma_{1,k,\lambda}} \Vert_{l^p U^2_{KP}} \Vert v_{\Gamma_{\lambda,k}} \Vert_{U^2_{KP}}. \]
The $L^4$ Strichartz estimate gives
\[ \begin{split}
\Vert   u_{\Gamma_{1,k,\lambda}}v_{\Gamma_{\lambda,k}}\Vert_{L^2}
\le & \sum_{\Gamma_{1,l}\subset\Gamma_{1,k,\lambda}} \Vert   u_{\Gamma_{1,l}}  v_{\Gamma_{\lambda,k}}\Vert_{L^2}
\\ \le & c\sum_{\Gamma_{1,l}\subset\Gamma_{1,k,\lambda}}\lambda^{\frac12}  \Vert   u_{\Gamma_{1,l}} \Vert_{U^4_{KP}} \Vert v_{\Gamma_{\lambda,k}}\Vert_{U^4_{KP}}
\\ \le & c \lambda^{\frac12}  \lambda^{2-\frac2p} \Vert   u_{\Gamma_{1,k,\lambda}} \Vert_{l^p U^4_{KP}} \Vert
v_{\Gamma_{\lambda,k}}\Vert_{U^4_{KP}}.
\end{split}
\]
where the summation is with respect to those $l$ for which $\Gamma_{1,l} \subset
\Gamma_{1,k, \lambda}.$
With the logarithmic interpolation of Lemma \ref{interpo}
 we arrive at
\[ \Vert   u_{\Gamma_{1,k,\lambda}}v_{\Gamma_{\lambda,k}}\Vert_{L^2}
\le c  \lambda^{1-\frac2p+\varepsilon}
\Vert u_{\Gamma_{1,k,\lambda}} \Vert_{l^pV^2_{KP}} \Vert v_{\Gamma_{\lambda,k}} \Vert_{V^2_{KP}}. \]

The summation with respect to $k$ is trivial and we arrive at the
first estimate \eqref{eq:3.10}, also for $L>1$, for which there
are only the obvious modifications, up to an explanation why we may
simply drop the operator $T_L$ once we restricted the support of the
Fourier transforms of the factors.  Bounded spatial Fourier
multipliers define bounded operators on the function spaces $U^p_{KP}$
and $V^p_{KP}$. Our problem is that $T_L$ is a bilinear Fourier
multiplier, and we have to reduce the estimates to estimates of
Fourier multipliers acting on single functions.  We recall that
\[ \rho_L(\xi_1,\eta_1, \xi_2,\eta_2) :=
\psi_L\left( \frac{\frac{\eta_1}{\xi_1}-\frac{\eta_2}{\xi_2}}{\xi_1+\xi_2}    \right)
\]
and we want to bound
$T_L( u_{\Gamma_{\mu,k,L\lambda/\mu}} , u_{\Gamma_{\lambda,k',L}})$
which is zero unless $4\le |k-k'|_\infty \le 20$. Without loss of generality we
consider $64\le k_1-k_1' \le 1000$. We apply a Galilee transform which reduces the problem to $k_1+k_1'=0$, $k_2=0$ and $|k_2'| \le 20$.
More precisely we expand
\begin{equation}
 u_{\Gamma_{\mu,k,L\lambda/\mu}} = \sum_{l\in A} u_{\Gamma_{\mu,l}} \end{equation}
where $A$ is set of cardinality $(L\lambda/\mu)^2$. The function $\rho_L$
is a smooth function on
$ \Gamma_{\mu,k,L\lambda/\mu} \times \Gamma_{\lambda,k',L}$.
We choose a smooth extension supported in
\[
\begin{split}
(( -3\mu,-\frac43\mu) \cup (\frac43 \mu, 3\mu))
\cup  \{ |\eta -kL\lambda \mu|_\infty \le  L\lambda \mu \}
& \\ & \hspace{-6cm} \times  (( -3\lambda,-\frac43\lambda) \cup (\frac43 \lambda, 3\lambda))
\cup  \{ |\eta -kL\lambda^2|_\infty \le  L\lambda^2 \},
\end{split} \]
which, by an abuse of notation, we call again $\rho_L$.
Its derivative satisfies
\[ \left| \partial_{\xi_1}^k \partial_{\eta_1}^\alpha \partial_{\xi_2}^l \partial_{\eta_1}^\beta  \psi_1\left( \frac{\eta_1}{\xi_1} -\frac{\eta_2}{\xi_2} \right) \right| \le c \mu^{-k} \lambda^{-l} (L\mu \lambda)^{-|\alpha|}
(L\lambda^2)^{-|\beta|}.  \]
We expand it into a fast converging Fourier series and we multiply
it by a suitable smooth product cutoff function
\[
\begin{split}
 \rho_{L} =&  \sum_\alpha  \rho_1(\xi_1/\mu)   e^{2\pi i   \alpha_1   \xi_1/\mu}
\rho_2(\eta_1/L\lambda \mu)  e^{2\pi i  \eta_1  \alpha_2/(L\lambda \mu) }
\\ & \times \rho_3(\xi_2/\lambda) e^{2\pi i \alpha_3  \xi_2/\lambda}
\rho_4(\eta_2/(L\lambda^2)  e^{2\pi i \eta_2/(L\lambda^2) }
\\ =: &f^\alpha =  \sum_{\alpha}  a_\alpha f_1^\alpha(\xi_1) f_2^\alpha (\eta_1)
f_3^\alpha (\xi_2) f_4^\alpha ( \eta_2)
\end{split}
\]
with uniform bounded compactly supported functions $f^\alpha_j$ and
summable coefficients $a^\alpha$. It suffices to bound the operator
\[ \begin{split}
T_{f^\alpha} (u_{\Gamma_{\mu,k,L\lambda/\mu}},v_{\Gamma_{\lambda,k',L}} )
= & M_{f_1^\alpha f_2^\alpha} u_{\Gamma_{\mu,k,L\lambda/\mu}}M_{f_3^\alpha f_4^\alpha} v_{\Gamma_{\lambda,k',L\lambda^2}}
\\ = & \tilde u_{\Gamma_{\mu,k,L\lambda/\mu}} \tilde v_{\Gamma_{\lambda,k',L\lambda^2}}.
\end{split}
\]
where $M_f$ denotes the Fourier multiplier.
The bilinear estimate above, together with the  observation  that spatial Fourier multipliers define bounded operators on $U^p_{KP}$ and $V^p_{KP}$ completes the
 argument for the
first estimate \eqref{eq:3.10} if $\mu \le \lambda/4$. If $\mu > \lambda/4 $
we decompose $v_\mu = v_{<\lambda/4} + \sum_{\lambda/4 \le \rho \le \mu}v_\rho$ and apply
\eqref{eq:3.10} to the first term and \eqref{eq:3.15} (which we prove next)
to the remaining terms.

We turn to estimate \eqref{eq:3.15}. It suffices to prove the estimate for $\mu =1 \le  \lambda $. We begin again with $L=1$. As
above it suffices to consider a fixed number $k \in \mathbb{Z}^2$,
which we even may assume to be zero. The summation with respect to $k$
poses no difficulties. The $L^4$ Strichartz estimate implies
$    \Vert u_{\Gamma_{\lambda,k}}^2
\Vert_{L^2} \le c  \lambda\Vert u_{\Gamma_{\lambda, k}}  \Vert_{U^4_{KP}}^2$.
By H\"older's inequality for sequences and orthogonality
\[ \sum_{k}  \Vert (u_{\Gamma_{\lambda,k}} u_{\Gamma_{\lambda,k}})_{\Gamma_{1,k,\lambda}}  \Vert_{l^r(L^2)}
\le   c  \lambda^{(\frac2r-1)_++1}
\Vert u_{\Gamma_{\lambda,k}} \Vert_{l^pV^2_{KP}} \Vert u_{\Gamma_{\lambda,k}} \Vert_{l^pV^2_{KP}} .
\]
The condition
$\frac1r= \frac1p+\frac1q$ suffices for that summation. This time there will be
an important modification for large $L$. As above, if $k \ge 2$, by
the bilinear estimate of Theorem \ref{th:2.1}, and its consequences for $U^2_{KP}$,
\begin{equation}
\Vert u_{\Gamma_{\lambda,0}} v_{\Gamma_{\lambda,k,L}} \Vert_{L^2}
\le c \lambda L^{-1} \Vert u_{\Gamma_{\lambda,0}}\Vert_{U^2_{KP}}
 \Vert  v_{\Gamma_{\lambda,k,L}} \Vert_{U^2_{KP}}.
\end{equation}
As above we have to sum over $L^2$ terms which gives
\[
\Vert T_L (u_{\Gamma_{\lambda,k,L}}, v_{\Gamma_{\lambda,k',L}}) \Vert_{L^2}
\le c L^{1-\frac2p + (1-\frac2q)_+ + \varepsilon} \lambda  \Vert u_{\Gamma_{\lambda,k,L}} \Vert_{l^pV^2_{KP}}  \Vert v_{\Gamma_{\lambda,k,L}} \Vert_{l^q V^2_{KP}}.
\]
We complete the proof with the same type of approximation and summation as above.
\end{proof}

\section{Proof of the main theorem}\label{sketch}
\subsection{A simple proof with three flaws}
We begin with sketching an incomplete proof, attempting to get an iteration
argument work in a simpler and slightly larger space $X^0$ defined by the norm
\[ \Vert u \Vert_{X^0} = \sup_{\lambda>0} \left(   \lambda^{1/2}
\Vert u_\lambda \Vert_{V^{2}_{KP}} + \lambda^{-1} \Vert u_\lambda
\Vert_{\dot{X}^{0,1}}\right). \]
This will almost work, and we will
provide essential modifications which will complete the wellposedness argument.
Existence via the contraction mapping principle follows from the two estimates
\begin{equation}\label{eq:3.1}\lambda^{\frac12}  \left\Vert  \int_0^t S(t-s)\partial_x (uv)_\lambda ds \right\Vert_{V^2_{KP}} \le c \Vert u \Vert_{X^0} \Vert v \Vert_{X^0} \end{equation}
and
\begin{equation}\label{eq:3.2}\lambda^{-1}  \left\Vert  \int_0^t S(t-s) \partial_x (uv)_\lambda ds \right\Vert_{\dot{X}^{0,1}} \le c \Vert u \Vert_{X^0} \Vert v \Vert_{X^0}. \end{equation}

It is useful to observe that
\begin{equation}\label{eq:3.3}\lambda^{1/2} \Vert   u_\lambda \Vert_{V^{2}_{KP}} +  \lambda^{-1}
\Vert u_\lambda \Vert_{\dot{X}^{0,1}} \sim
\lambda^{1/2} \Vert u_\lambda^{\leq\lambda^3} \Vert_{V^2_{KP}}
+ \lambda^{-1} \Vert u_\lambda^{>\lambda^3} \Vert_{\dot{X}^{0,1}}.
\end{equation}
 This implies \eqref{eq:3.3}.

By scaling  it suffices to consider \eqref{eq:3.1} and \eqref{eq:3.2} for $\lambda=1$, and duality reduces
the two estimates to bounds for trilinear integrals
\begin{equation}\label{eq:3.4}
\int uvw_1dxdydt=\int_S\widehat{u}(\xi_1,\eta_1,\tau_1)\widehat{v}(\xi_2,\eta_2,\tau_2)\widehat{w_1}(\xi_3,\eta_3,\tau_3)d\mathcal{H}^8.
\end{equation}
for  $w_1\in V^2_{KP}\cup L^2.$
Here $S$ denotes the subspace of dimension $8$ given by
\[
\{\xi_1+\xi_2+\xi_3=0,
\eta_1+\eta_2+\eta_3=0,\tau_1+\tau_2+\tau_3=0\}
\]
 and $d\mathcal{H}^8$
denotes the $8$-dimensional Hausdorff measure on it.  On this subspace
\eqref{eq:2.6} becomes
\begin{equation}\label{eq:3.5}\tau_1-\omega_1+\tau_2-\omega_2+\tau_3-\omega_3=
-3\xi_1\xi_2\xi_3-\frac{\xi_1\xi_2}{\xi_3}\Big|\frac{\eta_1}{\xi_1}
-\frac{\eta_2}{\xi_2}\Big|^2.
\end{equation}

It has the following important interpretation: If $\tau_i = \xi_i^3 -
\eta_i^2/\xi_i$ for $i=1,2$ then $\Lambda \ge |\xi_1\xi_2(\xi_1+\xi_2)|$ where
\[
\begin{split}
\Lambda:= \left| (\tau_2-\tau_1) - (\xi_2-\xi_1)^3 + \frac{|\eta_2-\eta_1|^2}{\xi_2-\xi_1} \right| = & |\xi_1||\xi_2||\xi_1+\xi_2| \\ &  + \frac{|\xi_1||\xi_2|}{|\xi_1+\xi_2|}  \left| \frac{\eta_1}{\xi_1} - \frac{\eta_2}{\xi_2} \right|^2
\end{split}\]
  $\Lambda$ is  a function of $\xi_i$ and $\eta_i$.
We  decompose $u,v$ into dyadic pieces  according to the size of $\xi$'s and,
by  an abuse of notation
we choose a version which is constant on the sets of consideration.
We decompose $u_i = u_i^{>{\Lambda/3}}  + u_i^{\le \Lambda/3} $.
 Then the trilinear integral vanishes unless at least one term has high modulation since
$ \int u_1^{\le \Lambda/3}  u_2^{\le \Lambda/3} u_3^{\le \Lambda/3} dx dy dt = 0$.
  The Strichartz estimates
 give for $\lambda\ge 1$
\begin{equation}
\label{eq:3.6}
\begin{split}
 \int u_\lambda  v_\lambda  w_1 dx\, dy\, dt \le &
\Vert w_1 \Vert_{L^2} \Vert u_\lambda v_\lambda \Vert_{L^2}
\\ \le & C
\Vert w_1 \Vert_{L^2} \left(\lambda^{\frac12} \Vert u_\lambda \Vert_{V^2_{KP}}\right)
\left( \lambda^{\frac12} \Vert v_\lambda \Vert_{V^2_{KP}} \right)
\end{split}
\end{equation}
which yields by scaling and orthogonality of the Paley-Littlewood pieces
\[ \left\Vert  \partial_x \int_0^t S(t-s) u_\lambda v_\lambda ds \right\Vert_{\dot X^{0,1}}
\le C \left(\lambda^{\frac12} \Vert u_\lambda \Vert_{V^2_{KP}}\right)
\left( \lambda^{\frac12} \Vert v_\lambda \Vert_{V^2_{KP}} \right). \]
By the bilinear estimate of Theorem \ref{th:2.1} - see also \eqref{eq:2.10}
\begin{equation}
\label{eq:3.7}
\begin{split}
 \left| \int u^{>\lambda^2/3}_\lambda  v_\lambda  w_1 dx\, dy\, dt \right|\le &
\Vert u^{>\lambda^2/3}_\lambda  \Vert_{L^2} \Vert v_\lambda w_1 \Vert_{L^2}
\\ \le &
c  \lambda^{-1} \Vert  u^{>\lambda^3/3}_\lambda  \Vert_{V^2_{KP}}\Vert v_\lambda \Vert_{U^2_{KP}}
\Vert w_1 \Vert_{U^2_{KP}}
\end{split}
\end{equation}
 and hence
\[
\begin{split}
 \left\Vert  \partial_x \int_0^t S(t-s) (u_\lambda v_\lambda)_1 ds \right\Vert_{\dot X^{0,1}}
+\left\Vert  \partial_x \int_0^t S(t-s) (u_\lambda v_\lambda)_1 ds \right\Vert_{V^2_{KP}}
\hspace{-3cm}&  \\ \le &  c \lambda^{\frac12} \Vert u_\lambda \Vert_{U^2_{KP}}
 \lambda^{\frac12} \Vert v_\lambda \Vert_{U^2_{KP}}.
\end{split}
\]
 For $\mu \le 1$ we estimate using the Strichartz estimate \eqref{eq:2.8} for $p=q=4$ and the embedding $V^2_{KP} \subset U^4_{KP}$
\begin{equation}
\label{eq:3.8}
\begin{split}
 \int u^{>\mu/3}_\mu  v_1  w_1 dx\, dy\, dt \le &
\Vert u^{>\mu/3}_\mu  \Vert_{L^2} \Vert v_1 w_1 \Vert_{L^2}
\\ \le &
c(\mu^{-1}  \Vert  u_\mu  \Vert_{\dot{X}^{0,1}}) \Vert v_1 \Vert_{V^2_{KP}} \Vert w_1 \Vert_{V^2_{KP}}
\end{split}
\end{equation}
and the bilinear estimate \eqref{eq:2.10} to arrive at
\begin{equation}
\label{eq:3.9}
\begin{split}
 \int u_\mu  v^{>\mu/3}_1  w_1 dx\, dy\, dt \le & c
\mu \Vert u_\mu  \Vert_{U^2_{KP}} \mu^{-\frac12} \Vert v_1 \Vert_{V^2_{KP}}  \Vert     w_1 \Vert_{U^2_{KP}}
\\ = & c (\mu^{1/2} \Vert u_\mu \Vert_{U^2_{KP}})\Vert v_1 \Vert_{V^2_{KP}}  \Vert     w_1 \Vert_{U^2_{KP}},
\end{split}
\end{equation}
thus
\[
\begin{split}
 \left\Vert  \partial_x \int_0^t S(t-s) (u_\mu v_1)_1 ds \right\Vert_{\dot X^{0,1}}
+\left\Vert  \partial_x \int_0^t S(t-s) (u_\mu v_1)_1 ds \right\Vert_{V^2_{KP}}
\hspace{-6cm}&  \\ \le &  c\left( \mu^{\frac12} \Vert u_\mu \Vert_{U^2_{KP}} + \mu^{-1} \Vert u_\mu \Vert_{\dot X^{0,1}}\right)
 \Vert v_1 \Vert_{U^2_{KP}}.
\end{split}
\]

To achieve \eqref{eq:3.1} and \eqref{eq:3.2}, there are three issues to resolve:
\begin{enumerate}
\item The summability with respect to $\lambda$ and $\mu$ requires improved estimates to obtain \eqref{eq:3.1} and \eqref{eq:3.2}.
\item In \eqref{eq:3.7} and \eqref{eq:3.9}, we have to replace $U^2_{KP}$ by $V^2_{KP}$.
\item The function $u=S(t) u_0$ for $t>0$ and $u=0$ for $t <0$ is not
in $\dot{X}^{0,1}$. We need a variant of the estimates for solutions
to the homogeneous initial value problem.
\end{enumerate}

Here as always we oversimplify things a bit: We have to consider more general
frequency combinations, and we only know that the two highest frequencies have to be of comparable size, otherwise the trilinear integral vanishes, which as always we ignore since we want to keep the formulas simpler, and there is
no new difficulty connected  with that.

\subsection{$l^p$ summation and bilinear estimate}\label{lp-summation}

We begin to explain  the modifications for the proof.
We use  $l^ql^p(V^2_{KP})$ with  $1\leq q\leq\infty,1<p<2$
 and replace $\dot{X}^{0,1}$ by $\dot{X}^{0,b}$ with some $b \in (\frac56,1)$
as discussed in the introduction.

\begin{definition}
Let  $X$ be  the space of all distributions
for which
 \[ \Vert u \Vert_{X} :=
\left\Vert\lambda^{\frac12} \Vert u_\lambda \Vert_{l^p V^2_{KP}}
+  \lambda^{2-3b} \Vert u_\lambda \Vert_{ \dot X^{0,b}}\right\Vert_{l^q_\lambda}
< \infty.   \]
\end{definition}

We next formulate  a bilinear estimate.
\begin{proposition}[Bilinear estimates for the quadratic term]
\label{pr:finbin}  For $u,v\in X$, we have
\begin{equation} \label{finbin}  \left\Vert \int_{-\infty}^t S(t-s) \partial_x(uv) ds \right\Vert_{X}
\le c \Vert u \Vert_X \Vert v \Vert_X.   \end{equation}
\end{proposition}

In our proof  we obtain a slightly  stronger bilinear estimate. We will replace the $U^2_{KP}$ by $V^2_{KP}$ at several places.

\begin{proof}
Using a Littlewood-Paley decomposition,  a duality argument
and an expansion of \eqref{finbin}  the estimate   follows from the next four inequalities. The high $\times $ high to low type estimates are
\begin{equation}
\label{first}
\begin{split}
 \int u_\lambda  v_\lambda  w_\mu dx\, dy\, dt &   \le C
\mu^{3b-3}\Big(\frac{\mu}{\lambda}\Big)^{2-2b -\varepsilon}   \\ & \hspace{-1cm}  \times \Big(\lambda^{\frac12} \Vert u_\lambda \Vert_{l^pV^2_{KP}} \Big)
  \Big( \lambda^{\frac12} \Vert v_\lambda \Vert_{l^pV^2_{KP}} \Big) \Vert w_\mu \Vert_{\dot X^{0,1-b}}
\end{split}
\end{equation}
\begin{equation}
\label{second}
\begin{split}
 \int u_\lambda  v_\lambda  w_\mu dx\, dy\, dt  &   \le C
\mu^{-\frac32}\Big(\frac{\mu}{\lambda}\Big)^{2-\frac2p-\varepsilon}\\ & \hspace{-1cm} \times\Big(\lambda^{\frac12} \Vert u_\lambda \Vert_{l^pV^2_{KP}} \Big)
  \Big( \lambda^{\frac12} \Vert v_\lambda \Vert_{l^pV^2_{KP}} \Big) \Vert w_\mu \Vert_{l^{p'}V^2_{KP}}.
\end{split}
\end{equation}
which we complement by low  $ \times $ high to high estnates
\begin{equation}
\label{third}
\begin{split}
 \int u_\mu  v_\lambda  w_\lambda  dx\, dy\, dt    \le  &
c \lambda^{-\frac32}\Big(\frac{\mu}{\lambda}\Big)^{\min\{ \frac2p-1-\varepsilon, 2b-\frac53 \} }
\\ &\hspace{-3cm} \times  \left( \mu^{1/2}   \Vert  u_\mu  \Vert_{l^pV^2_{KP}}+ \mu^{2-3b} \Vert u_\mu \Vert_{\dot{X}^{0,b}} \right) \Big(\lambda^{\frac12}\Vert v_\lambda \Vert_{l^pV^2_{KP}}\Big) \Vert w_\lambda \Vert_{l^{p'}V^2_{KP}}
\end{split}
\end{equation}
\begin{equation}
\label{fourth}
\begin{split}
 \int u_\mu  v_\lambda  w_\lambda  dx\, dy\, dt  &   \le
c \lambda^{3b-3}\Big(\frac{\mu}{\lambda}\Big)^{\min\{b-\frac12-\varepsilon,3b-\frac{5}{2}\}}
\\ &  \hspace{-3cm}\times  \left( \mu^{1/2}  \Vert  u_\mu  \Vert_{l^pV^2_{KP}}+ \mu^{2-3b} \Vert u_\mu \Vert_{\dot{X}^{0,b}} \right) \Big(\lambda^{\frac12}\Vert v_\lambda \Vert_{l^pV^2_{KP}}\Big) \Vert w_\lambda \Vert_{\dot X^{0,1-b}}
\end{split}
\end{equation}
for $\mu \le \lambda$. Proposition \ref{pr:finbin} and more precisely \eqref{finbin} follows by  summing  up the $\mu$ and $\lambda$, which is trivial.
More precisely we would have to consider
frequencies $\lambda_1 $ and $\lambda_2$ for the first estimates, but, since
on the Fourier side the Fourier variables $\xi_1$ and $\xi_2$ have to add up
to something of size $\sim \mu$ which we assume always less then $\lambda$, it suffices to consider neighboring dyadic
intervals resp $\lambda_1 \sim \lambda_2$. To simplify the notation we restrict
to $\lambda_1 = \lambda_2 = \lambda$ and we deal similarly with the other inequalities.

We  turn to the proof of     the four main estimates \eqref{first}-\eqref{fourth}.  For the [(high,high)$\to $ low]  type estimates \eqref{first} and \eqref{second}, by rescaling, we assume that $\mu=1$. We decompose
\[   \int u_\lambda  v_\lambda  w_1 dx\, dy\, dt
= \sum_{L \in 2^{\mathbb{N}}}   \int T_L(u_\lambda v_\lambda) w_1 dx \, dy \,dt \]
where the sum runs over $L=2^{\mathbb Z_+}$.

At least one of the terms has to have high modulation, i.e. modulation at least  $\ge L^2 \lambda^2/3$. For simplicity we will ignore the denominator $3$.  Now,
if $L> 1$ - the difference for $L=1$ is only in notation -
\begin{equation}\label{w-high}
\begin{split}
 \left|\int T_L(u_\lambda, v_\lambda) w_1^{\ge L^2 \lambda^2} dx dy dt \right|
\le \hspace{-3cm} &\hspace{3cm}    \Vert T_L (u_\lambda, v_\lambda)_1 \Vert_{l^2L^2} \Vert w_1^{\ge L^2 \lambda^2} \Vert_{l^2 L^2}  \\
\le & C \lambda  \Vert u_\lambda \Vert_{l^2V^2_{KP}}
 \Vert v_\lambda \Vert_{l^2V^2_{KP}}  (L\lambda)^{-2(1-b)}  \Vert w_1 \Vert_{l^2\dot{X}^{0,1-b}}.
\end{split}
\end{equation}

Since for $1<p<2$,
\[\Vert w_1\Vert_{l^2 \dot{X}^{0,1-b}}\approx \Vert w_1 \Vert_{\dot{X}^{0,1-b}}, \Vert u_\lambda\Vert_{l^2V^2_{KP}}\le \Vert u_\lambda\Vert_{l^p V^2_{KP}}\]
we obtain
\[
\begin{split}
\sum_{L} \left|\int T_L(u_\lambda v_\lambda) w_1^{\ge L^2 \lambda^2} dx dy dt \right|
& \\ & \hspace{-4cm} \le c \lambda^{2b-2}  \left(\lambda^{1/2} \Vert u_\lambda \Vert_{l^pV^2_{KP}} \right)
\left( \lambda^{1/2}  \Vert v_\lambda \Vert_{l^pV^2_{KP}} \right)   \Vert w_1 \Vert_{\dot{X}^{0,1-b}}.
\end{split}
\]
\eqref{w-high} can also be bounded, for $1<p<2$, by
\[ \lambda (L\lambda)^{\frac2p-1} \Vert u_\lambda \Vert_{l^pV^2_{KP}}
 \Vert v_\lambda \Vert_{l^pV^2_{KP}}  (L\lambda)^{-1}  \Vert w_1 \Vert_{l^{p'}V^2_{KP}}.
\]
Here we used H\"older's inequality and then the high modulation estimate for $w$ and \eqref{eq:3.15} with $r=q=p$
for the product. We complete the proof of \eqref{first} for the case the $w$ has high  modulation by
\[
\begin{split}
\sum_{L} \left|\int T_L(u_\lambda v_\lambda) w_1^{\ge L^2 \lambda^2} dx dy dt \right|
& \\ & \hspace{-4cm} \le c \lambda^{\frac2p-2}  \left(\lambda^{1/2} \Vert u_\lambda \Vert_{l^pV^2_{KP}} \right)
\left( \lambda^{1/2}  \Vert v_\lambda \Vert_{l^pV^2_{KP}} \right)   \Vert w_1 \Vert_{l^{p'}V^2_{KP}}.
\end{split}
\]
Next we use the symmetry property of Lemma \ref{TL} to deal with the case
that $v$ has high modulation:
\begin{equation}\label{w-low}\begin{split}
 \left|\int T_L(u_\lambda, v_\lambda^{\ge L^2\lambda^2})_1  w_1^{< L^2\lambda^2} dx dy dt \right|
= \hspace{-3cm} & \hspace{3cm}  \left|\int  v_\lambda^{\ge L^2\lambda^2} T_L (u_\lambda, w_1^{L^2\lambda^2})_\lambda dx dy dt \right|
\\
\le &   \Vert T_L (u_\lambda, w_1^{<L^2\lambda^2})_\lambda \Vert_{l^2L^2} \Vert v_\lambda^{\ge L^2 \lambda^2}  \Vert_{l^{2} L^2}  \\
\le & C  \lambda^{\frac2p-1+\varepsilon} L^\varepsilon \Vert u_\lambda \Vert_{l^pV^2_{KP}}
 \Vert w_1 \Vert_{l^{p'}V^2_{KP}}  (L\lambda)^{-1}   \Vert v_\lambda  \Vert_{l^{p}V^2_{KP}}.
\end{split}
\end{equation}
with the obvious modification if $L=1$. Here we used the high modulation estimate for $v_\lambda$
and \eqref{eq:3.10} with $r=2$, and $q=p'$.
The summation with respect to $L$ gives
\[
\begin{split}
\sum_{L}  \left|\int T_L(u_\lambda v_\lambda^{\ge L^2\lambda^2}) w_1^{< L^2\lambda^2} dx dy dt \right|
& \\ & \hspace{-4cm} \le
C  \lambda^{\frac2p-3+\varepsilon} \Big( \lambda^{1/2} \Vert u_\lambda \Vert_{l^pV^2_{KP}}\Big)
  \Vert w_1 \Vert_{l^{p'}V^2_{KP}}  \Big( \lambda^{1/2}    \Vert v_\lambda  \Vert_{l^pV^2_{KP}}\Big).
\end{split}
\]
In the same way, we can bound \eqref{w-low} by
\[\lambda^{\varepsilon} \Vert u_\lambda \Vert_{l^2V^2_{KP}}
 \Vert w_1^{<L^2\lambda^2} \Vert_{l^2V^2_{KP}}  (L\lambda)^{-1}   \Vert v_\lambda  \Vert_{l^2V^2_{KP}}.\]
Notice that\[ \Vert f^{\le L^2\lambda^2} \Vert_{V^2_{KP}}
 \lesssim  L^{2b-1}\lambda^{2b-1} \Vert f \Vert_{\dot{X}^{0,1-b}}.\]
\eqref{first} and \eqref{second} follows by a trivial summation over  $L$.

Now we turn to \eqref{third} and \eqref{fourth} and  rescale to $\lambda=1$. We  decompose the factors
in the same fashion as above
\[ \int u_\mu v_1 w_1 dx dy dt = \sum_L \int T_L(u_\mu v_1) w_1 dx dy dt .\]
As above, using \eqref{eq:3.10} with $r=q=p=2$
\[\begin{split} \left| \int T_L (u_\mu v_1) w_1^{\ge \mu L^2} dx dy dt \right|
\le & \Vert T_L(u_\mu v_1)_1 \Vert_{l^2 L^2}
\Vert w_1^{\ge \mu L^2} \Vert_{l^2L^2}
\\ & \hspace{-4.5cm} \le  C \mu^{\frac12}  (L/\mu)^{\varepsilon}  (\mu L^2)^{b-1}
(\mu^{1/2} \Vert u_\mu \Vert_{l^2V^2_{KP}}) \Vert v_1 \Vert_{l^2V^2_{KP}}
  \Vert w_1 \Vert_{l^2 \dot X^{0,1-b}}
\end{split}
\]
resp. taking $r=p=q<2$,
\[\begin{split} \left| \int T_L (u_\mu v_1) w_1^{\ge \mu L^2} dx dy dt \right|
\le & \Vert T_L(u_\mu v_1)_1 \Vert_{l^p L^2}
\Vert w_1^{\ge \mu L^2} \Vert_{l^{p'}L^2}
\\ & \hspace{-4cm} \le  C (L/\mu)^{1-\frac2p+\varepsilon}  L^{-1}
\mu^{1/2} \Vert u_\mu \Vert_{l^pV^2_{KP}} \Vert v_1 \Vert_{l^pV^2_{KP}}
  \Vert w_1 \Vert_{l^{p'} V^2_{KP}}.
\end{split}
\]

The summation with respect to $L$ gives
\[ \begin{split} \sum_L  \left| \int T_L (u_\mu v_1) w_1^{\ge \mu L^2} dx dy dt \right|
&  \le C \mu^{b-\frac12-\varepsilon}\\ & \hspace{-3cm}
\times\left( \mu^{1/2} \Vert u_\mu \Vert_{l^pV^2_{KP}}\right)
 \Vert v_1 \Vert_{l^p V^2_{KP}}
 \Vert w_1 \Vert_{\dot X^{0,1-b}}.
\end{split}
\]
resp.
\[
\begin{split}
 \sum_L  \left| \int T_L (u_\mu v_1) w_1^{\ge \mu L^2} dx dy dt \right| & \le  C \mu^{\frac2p-1-\varepsilon}\\ & \hspace{-3cm}
\times\left( \mu^{1/2} \Vert u_\mu \Vert_{l^pV^2_{KP}}\right)
 \Vert v_1 \Vert_{l^p V^2_{KP}}
 \Vert w_1 \Vert_{l^{p'}V^2_{KP}}.
\end{split}
\]

The same computation gives
\[
\begin{split}
\sum_L  \left| \int T_L (u_\mu w_1) v_1^{\ge \mu L^2} dx dy dt \right|
&  \le C  \mu^{\frac2p-\frac12-\varepsilon}\\ &
\hspace{-4cm}\times\left( \mu^{1/2} \Vert u_\mu \Vert_{l^pV^2_{KP}}\right)
 \Vert v_1 \Vert_{l^pV^2_{KP}}\Vert w_1 \Vert_{l^{p'}V^2_{KP}}
\end{split}
\]
resp.
\[
\begin{split}
\sum_L  \left| \int T_L (u_\mu w_1^{\leq \mu L^2}) v_1^{\ge \mu L^2} dx dy dt \right|
&  \le  C \mu^{b-\frac12-\varepsilon}\\ &
\hspace{-4cm}\times\left( \mu^{1/2} \Vert u_\mu \Vert_{l^2V^2_{KP}}\right)
 \Vert v_1 \Vert_{l^2V^2_{KP}}
\Vert w_1\Vert_{\dot X^{0,1-b}}.
\end{split}
\]
Here we used
\[ \Vert w^{<\mu L^2}_1 \Vert_{V^2_{KP}} \le C \mu^{b-\frac12}L^{2b-1}  \Vert w_1 \Vert_{\dot X^{0,1-b}}. \]
The last term with the high modulation on $u_\mu$ is different, and it is the most interesting:
\[ \left| \int u_\mu^{\ge \mu L^2}  T_L ( v_1, w_1)  dx dy dt \right|
\le  \Vert T_L( v_1, w_1)_\mu  \Vert_{l^{2} L^{2}_tL^{3/2}_{xy}}
\Vert u_\mu^{\ge \mu L^2} \Vert_{l^{2}L^2_t L^3_{xy}} . \]
We continue with the endpoint Strichartz estimate
\[
 \Vert  u_\mu^{\ge \mu L^2}  \Vert_{L^2_t L^3} \le   \Vert u_\mu^{\ge \mu L^2}
 \Vert_{L^2}^{1/2}
\Vert u_\mu^{\ge \mu L^2}  \Vert_{L^2_tL^6}^{1/2}
 \le C  (\mu L^2)^{\frac14-b} \mu^{\frac1{12}} \Vert u_\mu \Vert_{\dot X^{0,b}}
\]
for each part localized in $\eta$ and we achieve
\[\begin{split}  \left| \int u_\mu^{\ge \mu L^2}  T_L ( v_1 w_1)  dx dy dt \right|
&  \\ &\hspace{-4cm} \le C \mu^{2b  -\frac53 } L^{\frac12-2b+\frac2p-1} \Vert  v_1 \Vert_{l^p L^4_tL^3} \Vert  w_1 \Vert_{l^{p'} L^4_tL^3}
\mu^{2-3b} \Vert u_\mu \Vert_{\dot X^{0,b}}.
\end{split}
\]
By  Proposition \ref{pro:3.1}, we drop $T_L$ here.
The exponent $(4,\ 3)$ is a Strichartz pair. The summation with respect to $L$ is trivial. It gives
\[
\begin{split}
\sum_L \left| \int u_\mu^{\ge \mu L^2}  T_L ( v_1, w_1)  dx dy dt \right|
& \\ & \hspace{-3cm}  \le  c   \mu^{2b-\frac5{3}}      \left( \mu^{2-3b} \Vert u_\mu \Vert_{l^p\dot X^{0,b}}\right)
      \Vert v_1 \Vert_{l^p V^2_{KP}} \Vert w_1 \Vert_{l^{p'} V^2_{KP}}.
\end{split}
\]
resp.
\[
\begin{split}
\sum_L \left| \int u_\mu^{\ge \mu L^2}  T_L ( v_1, w_1^{\leq \mu L^2})  dx dy dt \right|
& \\ & \hspace{-3cm}  \le  c   \mu^{3b-\frac{5}{2}}      \left( \mu^{2-3b} \Vert u_\mu \Vert_{\dot X^{0,b}}\right)
      \Vert v_1 \Vert_{V^2_{KP}} \Vert w_1 \Vert_{ \dot{X}^{0,1-b}}.
\end{split}
\]
The summation with respect to
$\mu$ requires
$  b > \frac56$ and
we arrive at \eqref{third} and \eqref{fourth}  .
\end{proof}

\subsection{The initial data, the proof of wellposedness}\label{initial data}

It remains to consider estimate $S(t) u_0$ in terms of the initial data.
Let $$\tilde{u}(t)=\chi_{[0,\infty)}S(t)u_0.$$ As we pointed out in
issue iii), it is not in $\dot{X}^{0,b}$ for any $1/2<b\leq1$, thus it
is not in $X$ unless it is trivial.
Let
\[ \Vert u \Vert_Y= \left\Vert\lambda^{1/2} \Vert u_\lambda \Vert_{l^p U^1_{KP}} \right\Vert_{l^q_{\lambda}}\]
to shorten the notation. Then by construction
\[ \Vert \tilde u \Vert_{Y} \le \Vert u_0 \Vert_{ l^ql^pL^2}. \]

The two estimates of the following
proposition will allow to complete the proof.

\begin{proposition} \label{initial}
The following estimates hold.
\begin{equation} \label{eq:3.23}
\left\Vert \int_0^t S(t-s)\partial_x (uv) \right\Vert_X \le c  \Vert u   \Vert_Y  \Vert v  \Vert_{Y},
\end{equation}
\begin{equation} \label{eq:3.24}
\left\Vert \int_0^t S(t-s)\partial_x (uv) \right\Vert_X \le c
 \Vert u  \Vert_{X}  \Vert v  \Vert_{Y} .
\end{equation}
\end{proposition}

 With these estimates at hand we
 complete the fixed point argument.  By Duhamel's formula, to solve \eqref{eq:1.1} on $[0,\infty) $ is equivalent  to solving
\[w=\tilde u+\int_0^t S(t-s)\partial_x(w^2)(s)ds.\]
 We rewrite this equation in terms of the difference  $u=w- \tilde u$ and define the map
\begin{equation}\label{eq:3.26}
\begin{split}
\Phi(u):= & \int_0^t S(t-s)\partial_x((u+\tilde{u})^2)(s)ds
\\ = & \int_0^t  S(t-s) \partial_x\tilde{u}^2 ds + \int_{0}^t S(t-s) \partial_x (2 \tilde u u +u^2)ds
\end{split}
\end{equation}
where we set $u(s)=0 $ for $s<0$.

Set $r:=\min(\frac{1}{4C},3\varepsilon)$. Here $C$ is the largest constant among the constants from  \eqref{finbin}, \eqref{eq:3.23} and \eqref{eq:3.24}. We define the closed ball of radius $r$ in $X$
\[ B_r:=\{u\in X ; \|u\|_{X}\leq r\}.\]
We search  an unique  fixed point of $\Phi$ in $B_r$.
By the definition of $Y$
\[ \|\tilde u  \|_{Y}\leq C  \|u \|_{l^ql^p L^2}\lesssim \varepsilon.\]
By  \eqref{finbin}, \eqref{eq:3.23} and \eqref{eq:3.24}, we have

\begin{equation}\label{eq:3.27}\|\Phi(u)\|_{X}\lesssim \|u\|_{X}^2+2\|\tilde u  \|_{Y}\|u\|_{X}+\|\tilde u  \|_{Y}^2\leq r
\end{equation}
and
\[
\begin{split}
\|\Phi(u)-\Phi(v)\|_{X} \le &   C\|u-v\|_{X}(\|u\|_{X}+\|v\|_{X} +\|\tilde{u}\|_{Y})
\\ \leq & \frac12\|u-v\|_{X}.
\end{split}\]
We apply the contraction mapping theorem to obtain existence of a unique fixed
point. The linearization at the fixed point is invertible - it is a contraction
by construction - and the map $\Phi$ is analytic. Hence the map from the
initial data to the fixed point is analytic.

The estimate
$$\|u\|_X\leq C\|u_0\|^2_{l^ql^pL^2}$$
follows from \eqref{eq:3.27}. This completes the proof, up to proving Proposition \ref{initial}.

\subsection{The proof of Proposition  \ref{initial}} \label{proof for initial data}

By the same strategy as above
 we continue to assume $\mu \le 1 \le \lambda$. The estimates \eqref{first}
and \eqref{second} are in terms of $l^pV^2_{KP}$ at frequency $\lambda$. It is
a consequence of Minkowski's inequality  that
\[ \Vert u_\lambda \Vert_{l^pV^2_{KP}} \lesssim \Vert u_\lambda \Vert_{l^pU^2_{KP}}\lesssim \Vert u_\lambda \Vert_{l^p U^1_{KP}} .\]
We can directly replace
$l^pV^2_{KP}$ by $l^pU^1_{KP}$ in the estimates \eqref{first} and \eqref{second}.
This completes the argument for the [(high,high)$\to $ low]  case, for both estimates \eqref{eq:3.23}  and \eqref{eq:3.24}.  The next  lemma provides the remaining [(low,high) $\to $ high] estimates.

\begin{lemma}\label{U1U1}  The following estimates hold, for $\mu\le 1$,
\begin{equation} \label{v2kpii2}
\begin{split}
\left\| \int_0^t S(t-s)  (u_\mu v_1 )_1ds \right\|_{l^p(U^2_{KP})}
& \\ & \hspace{-4cm} \le C \mu^{\min(\frac{2}{p}-1-\varepsilon,b-\frac12)}\Big( \mu^{1/2}    \Vert u_\mu \Vert_{l^pU^1_{KP}}\Big)
\Vert v_1 \Vert_{l^pU^1_{KP}},
\end{split}
\end{equation}
\begin{equation} \label{itfirstl22}
\left\|\int_0^t S(t-s)(u_\mu v_1 )_1 \right\|_{ \dot{X}^{0,b}}
\le C \mu^{\min(\frac2p-1-\varepsilon,b-\frac12)} \Big( \mu^{\frac12}   \Vert u_\mu \Vert_{l^pU^1_{KP}}\Big)
\Vert v_1 \Vert_{l^pV^2_{KP}}.
\end{equation}
\begin{equation} \label{itfirstl23}
\left\|\int_0^t S(t-s)(u_\mu v_1 )_1 \right\|_{ \dot{X}^{0,b}}
\le C \mu^{\min(\frac2p-1-\varepsilon,b-\frac12)} \Big( \mu^{\frac12}   \Vert u_\mu \Vert_{l^pV^2_{KP}}\Big)
\Vert v_1 \Vert_{l^pU^1_{KP}}.
\end{equation}
 \end{lemma}

Together with the versions of \eqref{first} and \eqref{second} above these imply \eqref{eq:3.24} then \eqref{eq:3.23} in Proposition \ref{initial} by an easy summation.
\begin{proof}
Again we use  duality and decompose
\[
\left|\int u_\mu v_1 w_1 dx\,  dy\, dt\right| \le
\sum_L \left|\int u_\mu T_L( v_1, w_1)  dx\, dy\, dt\right|.
\]
At least one term has modulation $\geq \mu L^2$.  Notice that
\[\Vert u_\lambda\Vert_{l^p(V^2_{KP})}\lesssim \Vert u_\lambda\Vert_{l^p U^1_{KP}},\]the estimates in \eqref{third} and \eqref{fourth} work well except
 the case $u_{\mu}$ has the high modulation

\[
 \int u_\mu^{>\mu L^2}  T_L( v_1 w_1)  dx\,  dy\,  dt.
\]
Let  $L\ge 1$ and consider
\[ \int u^{>\mu L^2}_{\Gamma_{\mu,k,L/\mu}}   T_L (v_{\Gamma_{1,k',L}} ,u_{\Gamma_{1,k,L}}) dx\,  dy\,  dt \]
with $16 \le |k-k'|\le 1000$ if $L>1$, resp $|k-k'|\le 200$ if $L=1$.
we decompose $u_{\Gamma_{\mu,k,\frac{L}{\mu}}}$ further
\[
\begin{split}
\sum_{16\leq|k-k'|\leq100} \sum_{|k-l|\lesssim\frac{L}{\mu}}\left|\int u^{>\mu L^2}_{\Gamma_{\mu, l}}   T_L  (v_{\Gamma_{1,k,L}} , w_{\Gamma_{1,k,L}})   dx\, dy\, dt \right|
\hspace{-7.5cm} & \\ \lesssim &\sum_{16\leq|k-k'|\leq100}\sum_{|k-l|\lesssim\frac{L}{\mu}}
 \Vert u^{>\mu L^2}_{\Gamma_{\mu, l}}    \Vert_{L^1_tL^\infty}
\Vert  T_L (v_{\Gamma_{1,k',L}} , w_{\Gamma_{1,k,L}})  \Vert_{L^\infty_t L^1}
\\ \lesssim &  \sup_{k} \mu^{\frac52}\Big(\frac{L}{\mu}\Big)^{\frac{2}{p^\prime}}
\Vert u_{\Gamma_{\mu,k,\frac{L}{\mu}}} \Vert_{l^p(L^1_tL^2)}
\Vert v_1 \Vert_{l^p L^\infty_tL^2}  L^{\frac{2}{p}-1}\Vert w_1 \Vert_{l^{p'}L^\infty_tL^2}
\\ \lesssim &  \mu^{\frac2p-1}L^{-1}
\Big(\mu^{\frac12}\Vert u_\mu \Vert_{l^p U^1_{KP}}\Big)
\Vert v_1 \Vert_{l^p V^2_{KP}}  \Vert w_1 \Vert_{l^{p'}V^2_{KP}}.
\end{split}
\]
Here we used the size of the set $\Gamma_{\mu, l}$ is $\mu^5$.  We estimate similarly to above
\[
\begin{split}
 \sum_{16\leq|k-k'|\leq100}\left|\int u^{>\mu L^2}_{\Gamma_{\mu, k,L/ \mu}}   T_L  (v_{\Gamma_{1,k,L}} , w_{\Gamma_{1,k,L}})   dx\, dy\, dt \right|
\hspace{-5.5cm} & \\
\lesssim &  \mu^{b}L^{2b-2}
\Vert u_\mu \Vert_{l^2U^1_{KP}}
\Vert v_1 \Vert_{l^2 V^2_{KP}}  \Vert w_1 \Vert_{\dot{X}^{0,1-b}}.
\end{split}
\]
 Here we applied
Sobolev's resp. Bernstein's inequality in sets of Fourier size
 $\mu^3L^2$
and the  high modulation factor $ \mu L^2$. The summation with
respect to  $L$ is trivial since the exponent is negative.  Finally \eqref{itfirstl23}  is a direct consequence of \eqref{fourth}.
\end{proof}

\section{Ill-posedness and Function spaces}\label{ill}

\subsection{ Ill-posedness in $l^ql^pL^2$ for $p>2$. }

We prove illposedness (Theorem \ref{ill}) by contradiction. By scaling it suffices
to consider $T=1$. Suppose that the flow map $u_0 \to u(1)$ defines
a map from $l^ql^pL^2$ to itself which is continuously differentiable near $0$,
and twice differentiable at $0$,  for some $p>2$. For simplicity we choose
$q=\infty$, but the proof works for all $q\in [1,\infty] $.

 Consider the Cauchy problem
\begin{align}\label{eq:6.1}
\left\{
\begin{aligned}
&\partial_x\left(\partial_t u+\partial_x^3 u+\partial_x(u^2)\right)+\triangle_{y }u=0 \\
&u(0,x,y )=\gamma \phi(x,y )\  \  \gamma\in\mathbb R.
\end{aligned}
\right.
\end{align}
where $\phi\in l^\infty l^pL^2$ and $1<p<\infty$. Suppose that $u(\gamma,t,x,y )$ solves (\ref{eq:6.1}). By Duhamel's formula, we have
$$u(\gamma,t,x,y )=\gamma S(t)\phi(x,y )+\int_0^t S(t-s)\partial_x(u(\gamma,x,y )^2)(s)ds. $$
Since the flow map is (twice) differentiable at $u_0=0$
$$\frac{\partial u}{\partial \gamma}(0,t,x,y )=S(t)\phi(x,y ):=u_1(t,x,y ),$$
$$\frac{\partial^2u}{\partial\gamma^2}(0,t,x,y )=-2\int_0^t S(t-s)\partial_x(u_1^2(s))ds:=u_2(t,x,y ).$$
Since we assume the flow map to be twice differentiable
\begin{equation}\label{eq:6.2}
\|u_{2}(1,\cdot)\|_{l^\infty l^p L^2}\lesssim \|\phi\|^2_{l^\infty l^p L^2}.
\end{equation}
We construct a sequence of initial data for $u_1$ of norm $1$ so that
the norm of $u_2(1)$ tends to infinity. This yields the desired contradiction.

We define  the initial data $\phi$ defined by its Fourier transform
\[\begin{split}\hat{\phi}(\xi,\eta)=&\frac{1}{\mu^3\big(\frac{\lambda}{\mu}\big)^\frac2p}\chi_{[\frac{\mu}{2},\mu]}(\xi)\chi_{[\frac{\lambda\mu}{2},2\lambda\mu]^2}(\eta)+\frac{1}{\mu^{\frac32}\lambda^\frac32}\chi_{[\lambda+\frac{\mu}{2},\lambda+\mu]}(\xi)\chi_{[\frac{\lambda\mu}{2},2\lambda\mu]^2}(\eta)\\
:=&\hat{\phi}_1+\hat{\phi}_2.\end{split}\]
Here dyadic numbers $\mu\ll1\ll \lambda$ will be chosen later.  It is easy to check that
$$\|\phi\|_{l^\infty l^p L^2}\approx \|\phi_1\|_{l^\infty l^p L^2}\approx\|\phi_2\|_{l^\infty l^p L^2}\approx 1.$$
Moreover
$$u_1=S(t)\phi_1+S(t)\phi_2,$$
$$u_1^2=(S(t)\phi_1)^2+(S(t)\phi_2)^2+2(S(t)\phi_1S(t)\phi_2):=f_1+f_2+f_3.$$
The   Fourier transforms of the three summand are supported on pairwise disjoint sets and  they are orthogonal. We then decompose $u_2$ into three orthogonal parts as
$$u_2(1)=\int_0^1 S(1-s)(f_1+f_2+f_3)(s)ds:=F_1+F_2+F_3.$$
By (\ref{eq:6.2}), we have \begin{equation}\label{eq:6.3}\|F_3(1)\|_{l^\infty l^p L^2}\lesssim \|u_2(1,\cdot)\|_{l^\infty l^p L^2}\lesssim1.\end{equation}
By Lemma 4 in Page 376 of \cite{MoSaTz}, we have
\[
\hat{F}_3(1,\xi,\eta)=2\frac{\xi
  e^{i\omega(\xi,\eta)}}{\mu^3\big(\frac{\lambda}{\mu}\big)^{\frac2p}\mu^\frac{3}{2}\lambda^\frac32} \int_A  \frac{e^{iR(\xi,\xi_1,\eta,\eta_1)}-1}{R(\xi,\xi_1,\eta,\eta_1)}d\xi_1
d\eta_1.
\]
Here $R(\xi,\xi_1,\eta,\eta_1)$  denotes the resonance function
\begin{equation}
-3\xi\xi_1(\xi-\xi_1)-\frac{\xi\xi_1}{\xi-\xi_1}\Big|\frac{\eta}{\xi}
-\frac{\eta_1}{\xi_1}\Big|^2.
\end{equation}
In the set
\[
\begin{split}
A= & \Big\{ \xi_1,\eta_1: \xi_1\in[\frac{\mu}
    {2},\mu],\eta_1\in[\frac{\lambda\mu}{2},2\lambda\mu]^2,\\ &
  \xi-\xi_1\in[\lambda+\frac{\mu}{2},\lambda+\mu],\eta-\eta_1\in[\frac{\lambda\mu}{2},2\lambda\mu]^2\Big\}
\end{split}   \]
the resonance function is bounded from below:
$$|R(\xi,\xi_1,\eta,\eta_1)|\sim \lambda^2\mu.$$
If   $\mu\lambda^2=O(1)$ (we may choose $\mu$ and $\lambda$) and obtain
$$\frac{e^{iR(\xi,\xi_1,\eta,\eta_1)}-1}{R(\xi,\xi_1,\eta,\eta_1)}=1+O(1).$$
 It follows that
$$|\hat{F}_3(1,\xi,\eta)|\geq\frac{\lambda\lambda^2\mu^3}{\mu^3\big(\frac{\lambda}{\mu}\big)^{\frac2p}\mu^{\frac32}\lambda^{\frac32}}\chi_{[\lambda+\mu,\lambda+\frac{3\mu}{2}]}(\xi)\chi_{[\lambda\mu,2\lambda\mu]^2}(\eta).$$
Then
\begin{equation}\label{eq:6.4}1\gtrsim\|F_3\|_{l^\infty l^p L^2}\gtrsim\frac{\lambda^3\mu^3}{\mu^3\big(\frac{\lambda}{\mu}\big)^\frac{2}{p}}\approx\frac{\lambda^3}{\lambda^{\frac{6}{p}}}.\end{equation}
Here we used $\mu\lambda^2=O(1)$.   Since $\lambda\gg1$ we arrive at a contradiction to \eqref{eq:6.4}  unless $p\leq 2$.

\subsection{The function spaces $l^ql^pL^2$}

We prove Theorem \ref{discription}.  By the embedding $l^ql^pL^2 \subset l^{\tilde q} l^{\tilde p} L^2$ if
$\tilde q \ge q$ and $\tilde p \ge p$ it suffices to prove endpoint statements.

(i)
Let $f$ be a Schwartz function  and fix $\lambda$. Trivially
\[
\Big(\sum_{l\in \lz}\|f_{\Gamma_{\lambda,l}}\|_{L^2}^p\Big)^{\frac1p}= \Big(\sum_{M\geq \lambda^2}\sum_{|l|\sim\frac{M}{\lambda}}\|f_{\Gamma_{\lambda,l}}\|_{L^2}^p\Big)^{\frac1p}
\]
and for $\frac43 <p$ and $N>2$, we have
$$\|f_{\Gamma_{\lambda,l}}\|_{L^2}\lesssim_{N} \frac{\lambda^{\frac32}}{(1+\lambda+M)^{N}},$$
thus
\begin{equation} \label{schwartzdyadic}
\Big(\sum_{l\in \lz}\|f_{\lambda,\Gamma_{l,\lambda}}\|_{L^2}^p\Big)^{\frac1p}\lesssim_N\frac{\lambda^{\frac32-\frac2p}}{(1+\lambda)^N} \end{equation}
and for $p>2$
\[\sum_{\lambda}\lambda^{-\frac12}\|f_{\lambda}\|_{l^p(L^2)}<\infty.\]
By duality $l^\infty l^pL^2 (p<2)$  embeds into the space of  distributions.
A small modification shows that $l^1l^2L^2$ embedds into the space of distributions.

(ii) It suffices to   construct a sequence of Schwartz functions
which converges in $l^pl^2L^2$ ($p>1$) but diverges as distributions. Since
\[ l^pl^2L^2 = L^2(\R^2; \dot B^{\frac12}_{2,p} ) \]
it suffices to construct a sequence of functions $\phi_\mu$ of one variable of norm $1$ in  $\dot B^{\frac12}_{2,q}$ and a Schwartz function $\phi$ so that
$\int \phi_\mu \phi dx \to \infty$.
 Here $\dot B^{\frac12}_{2,q}$
denotes the homogeneous Besov space.
This is well known but we  give an example for completeness.
For $0<\lambda$ we choose  a
Schwartz function $f_\lambda$ with the property
\[\hat{f}_{\lambda}(\xi)=\left\{\begin{array}{ll}\lambda^{-1},&\text{for } |\xi|\sim \lambda,\\
0,&\text{\, otherwise.}\end{array}\right.\]
For any fixed $\mu\ll 1$, we define
\[\phi_\mu=\frac{1}{|\ln\mu|^{\frac1p}}\sum_{\mu^2\leq\lambda\leq\mu}f_\lambda.\]
 It is easy to see
$$\|\phi_{\mu}\|_{\dot B^{\frac12}_{2q}} \sim 1.  $$
However if   $\psi$ is a Schwartz function with Fourier transform supported
in the ball $B(0,2)$ and $\hat{\psi}=1$ in the unit ball $B(0,1)$ then
\[<\psi,\phi_\mu>\sim\sum_{\mu^2\leq\lambda\leq \mu}|\ln\lambda|^{-\frac1p}  \sim |\ln \mu|^{1-\frac1p} .\]

(iii) Suppose now that the   Schwarz function $\phi$ is in $l^\infty l^pL^2$ for $p<\frac43$.
 We assume there exists $(0,\eta_0)\in \mathbb R^3$ such that $\hat{\phi}(0,\eta_0)\neq0$. By continuity, there exists $r,c>0$ such that
\[
|\hat{\phi}(\xi,\eta)|>c,\,\,\text{for}\,\,(\xi,\eta)\in B:= B((0,\eta_0),r).
\]
 Then
\[
\begin{split}
\sup_{\lambda}\lambda^\frac12\Big(\sum_{ l\in\lz}\|\phi_{\Gamma_{\lambda, l}}\|_{L^2}^p\Big)^{\frac1p}&\ge \sup_{\lambda\lesssim r} \lambda^\frac12\Big(\sum_l\|\phi_{\Gamma_{\lambda,l}\cap B}\|_{L^2}^p\Big)^{\frac1p}\\
&\sim\sup_{\lambda\lesssim r}r^{\frac2p}\lambda^{1-\frac{2}{p}}\|\phi_{\lambda\cap B}\|_{L^2}\sim \sup_{\lambda\lesssim r} cr^{\frac2p+1}\lambda^{\frac32-\frac2p}.
\end{split}
\]
which is $\infty$ if  $1<p<\frac43$.  This is a contradiction and hence
\[ 0=\hat \phi(0,\eta)= (2\pi)^{-\frac32} \int e^{-iy\eta} \phi(x,y) dx dy    \]
 for all $\eta \in \R^2$. The conclusion for $l^q l^{\frac43}L^2$ follows in the same fashion.

(iv)  It follows from \eqref{schwartzdyadic}
that  Schwartz functions are contained in $l^\infty l^p L^2$ if
 $\frac43\leq p $ and in $l^ql^pL^2$ if $\frac43<p$
 and  $1\leq q<\infty$.

\bibliographystyle{plain}

%\bibliography{KPII}
%\printbibliography

\end{document}